\documentclass[aos, preprint]{imsart}%

\RequirePackage{natbib}

\arxiv{math.PR/0000000}

\usepackage{amsfonts}
\usepackage{epsfig}
\usepackage{amssymb,amsmath,bm}
\usepackage{subfig}
\usepackage{amsthm}
\usepackage{mathrsfs}
\usepackage{enumerate,color}
\usepackage{amsmath}
\usepackage{amssymb}
\usepackage{graphicx}%
\usepackage{natbib}
\providecommand{\U}[1]{\protect\rule{.1in}{.1in}}

\newtheorem{remark}{Remark}
\newtheorem{theorem}{Theorem}
\newtheorem{lemma}{Lemma}
\newtheorem{example}{Example}

\newtheorem{procedure}{Procedure}

\newenvironment{hypA}[1]{\begin{itemize}
\item[({\bf A#1})]}{\end{itemize}}
\newcounter{hypB}

\begin{document}

\begin{frontmatter}

\title{Asymptotic Behaviour of Approximate Bayesian Estimators}
\runtitle{Approximate Bayesian Computation}

\begin{aug}
  \author{\fnms{Thomas. A.}  \snm{Dean}\corref{}\thanksref{t2}\ead[label=e1]{tad36@cam.ac.uk}}
  \and
  \author{\fnms{Sumeetpal S.} \snm{Singh}\thanksref{t2}\ead[label=e2]{sss40@cam.ac.uk}}

  \thankstext{t2}{T.A. Dean and S.S. Singh's research is funded by the Engineering and Physical Sciences Research Council (EP/G037590/1) whose support is gratefully acknowledged. 
This}

  \runauthor{Dean et al.}

  \affiliation{University of Cambridge}

  \address{T.A.Dean, S.S. Singh, \\
  Department of Engineering, \\
  University of Cambridge, \\
  Cambridge, \\
  CB2 1PZ, UK\\
          \printead{e1,e2}}

\end{aug}

\begin{abstract}
Although approximate Bayesian computation (ABC) has become a popular technique for performing parameter estimation when the likelihood functions are analytically intractable there has not as yet been a complete investigation of the theoretical properties of the resulting estimators.  In this paper we give a
theoretical analysis of the asymptotic properties of ABC based parameter estimators for hidden Markov models and show that ABC based estimators satisfy asymptotically biased versions of the standard results in the statistical literature.
\end{abstract}



\end{frontmatter}

\bigskip


\thispagestyle{plain}

\section{Introduction}

One of the most fundamental problems in statistics is that of parameter estimation.  Suppose that one has a collection of probability laws $\mathbb{P}_{\theta}$ parametrised by a collection of parameter vectors $\theta \in \Theta$.  Suppose further that one has data $\hat{Z}$ generated by a process distributed according to some law $\mathbb{P}_{\theta^{\ast}}$ where the exact value of $\theta^{\ast} \in \Theta$ is unknown.  The problem of parameter estimation is to infer the value of the unknown parameter vector $\theta^{\ast}$ from the data $\hat{Z}$.  Many standard methods for estimating the value of $\theta^{\ast}$ are based upon using the
likelihood function $p_{\theta} (\hat{Z})$.  For example
Bayesian approaches use the likeilhood to reweight some prior distribution
to obtain a posterior distribution on the space of parameter vectors
that represents ones sense of certainty of any given parameter vector
being equal to $\theta^{\ast}$. Alternatively one may take a frequentist
approach and estimate $\theta^{\ast}$ with the parameter vector which
maximises the value of the corresponding likelihood (ie. maximum likelihood
estimation (MLE)). 

Of course these approaches all rely on one being
able to compute the likelihood functions $p_{\theta} (\hat{Z})$, either exactly or numerically.  However, in a wide range of applications this is not possible, either because no analytic expression for the likelihoods exists or else because computing them is computationally intractable.  Despite this  one is  often still
able, in such cases, to generate random variables distributed according to the corresponding laws $\mathbb{P}_{\theta}$.  This has led to the development of methods in which $\theta^{\ast}$ is estimated by implementing
a standard likelihood based parameter estimator using some principled
approximation to the likelihood instead of the true
likelihood function itself. In general these approximations are estimated using Monte Carlo simulation based on generating samples from the relevant probability distributions.

A method which has recently
become very popular in practice and on which we shall focus our attention for the rest of this paper is approximate Bayesian computation (ABC).  A non-exhaustive list of references for applications of the method includes: \citep{mckcoodea2009,petwutshe2010,priseiperfel1999,ratandwiuric2009,tavbalgridon1997}. See also \citep{sisfan2010} for a review
on computational methodology.  The standard ABC approach to approximating the likelihood is as follows.  Suppose that the distributions $\mathbb{P}_{\theta}$ all have a density  $p_{\theta}\left(\cdot\right)$ on some space $\mathbb{R}^{m}$ w.r.t. some dominating measure
$\mu$.  Furthermore suppose that the functions $p_{\theta} \left(\cdot\right)$ cannot
be evaluated directly but that one can generate random variables distributed
according to the laws $\mathbb{P}_{\theta}$. Given some data $\hat{Z}$ the general ABC approach to approximating
the values of the likelihood functions $p_{\theta} ( \hat{Z} )$  is to choose a metric $d\left(\cdot,\cdot\right)$
on $\mathbb{R}^{m}$ and a tolerance parameter $\epsilon>0$ and for all $\theta \in \Theta$ approximate
the likelihood $p_{\theta} ( \hat{Z} )$ with 
\begin{equation}
p^{\epsilon}_{\theta} (\hat{Z}) \triangleq  \mathbb{P}_{\theta} \left(d(\hat{Z},Z)\leq\epsilon\right).\label{secABCapproxdef}
\end{equation}
Typically the probabilities \eqref{secABCapproxdef}
are themselves estimated
using Monte Carlo techniques.  A particularly appealing feature of the ABC methodology is that, despite the methods name, the resulting approximations to the likelihoods may then be used in any likelihood based parameter inference methodology the user desires.

Intuitively, the justification for the ABC approximation is that for sufficiently small
$\epsilon$
\[
\frac{1}{ \mu \left( B_{\hat{Z}}^{\epsilon} \right) }\mathbb{P}_{\theta} \left(d(\hat{Z},Z)\leq\epsilon\right)
\approx  p_{\theta} \left(  \hat{Z} \right) 
\]
 where $B_{\hat{Z}}^{\epsilon}$ denotes the $d$-ball of radius $\epsilon$ around the point $\hat{Z}$ and thus the probabilities \eqref{secABCapproxdef} will provide a good
approximation to the likelihood, up to the value of some
renormalising factor which is independent of $\theta$ and hence can be ignored.

Clearly in general the estimators based on ABC approximations to the likelihood will differ from those based on the exact value of the likelihood function, however although the use of ABC has become commonplace there has to date
been little investigation of the precise nature of the theoretical properties of ABC based estimators.  One notable exception is \citep{feapra2010}.  In this paper the authors consider the problem of finding the optimal choice, for a given data set, of summary statistic and $\epsilon$ in order to  minimise the mean square error of the resulting ABC posterior distribution on parameter space.  Unfortunately the resulting optimal choice of summary statistic involves computing a conditional expectation w.r.t. the unknown posterior distribution and hence it can only be computed approximately and not exactly.  Further the analysis is done only for fixed size data sets and the asymptotic properties of the ABC estimator are left unexplored.

An alternative approach is taken in \citep{deasinjaspet2010} in which the asymptotic behaviour of the MLE implemented with the ABC approximation to the likelihood (henceforth ABC MLE) was studied.  The analysis in this paper is based on the observation that the ABC approximation to the likelihood can be considered as being equal to the likelihood function of a perturbed probability distribution.  Using this observation it was shown that ABC MLE in some sense inherits its behaviour from the standard MLE but that the resulting estimator has an innate asymptotic bias.  Furthermore, it is shown that this bias can be made arbitrarily small by choosing a sufficiently small values of the ABC parameter $\epsilon$. 

The results in \citep{deasinjaspet2010} concerning the asymptotic behaviour of ABC MLE provide a mathematical justification of this method analgous to that provided for the standard MLE by the results concerning asymptotic consistency.  However they do not establish any asymptotic normality type properties of this estimator and there  are as yet no analogous results for the ABC Bayesian parameter estimator.  The aim of this paper is to bridge these theoretical gaps by showing that the standard results in likelihood based parameter estimation, that is to say  asymptotic consistency, asymptotic normality and Bernstein-von Mises type theorems, also hold in a suitably modified version for parameter estimators based on ABC approximations to the likelihood.  In the next section we provide an outline of the approach that we shall take to proving these results.

\subsection{Contributions and Structure}

In this paper we shall study the asymptotic behaviour of ABC  parameter estimators when used to perform inference for hidden Markov models.  This will be convenient as (as we will show) the Markovian context imbues the ABC approximations with a particularly nice mathematical structure.  Furthermore, as HMMs are used as statistical models in a wide range of applications including
Bioinformatics (e.g.~\citep{dureddkromit1998}), Econometrics (e.g.~\citep{kimshechi1998}) and Population genetics (e.g.~\citep{felchu1996}) (see also \citep{caprydmou2005} for a recent overview), the class of models thus considered is sufficently general to be of genuine practical interest. 

For the purpose of this paper a HMM will be considered to be  a pair of
discrete-time stochastic processes, $\left\{  X_{k}\right\}  _{k \geq0}$ and
$\left\{  Y_{k}\right\}  _{k \geq0}$. The hidden process, $\left\{
X_{k}\right\}  _{k \geq0}$, is a homogenous Markov chain taking values in some
Polish space $\mathcal{X}$ and the observed process $\left\{  Y_{k}\right\}
_{k \geq0}$ takes values in $\mathbb{R}^{m}$ for some $m\geq1$. Conditional on
$X_{k}$ the observations $Y_{k}$ are statistically independent of the random
variables $Y_{0} , \ldots, Y_{k-1} ; X_{0} , \ldots, X_{k-1}$.   In many models the densities of the conditional laws of the observed process w.r.t.~the hidden state either have no known analytic expression or else are computationally intractable.  In this case it follows that standard methods to estimating the likelihoods of the observed process, eg.~SMC, can no longer be used and that an alternative approach like ABC must be used.   For the rest of this paper  we shall consider performing ABC based parameter estimation for HMMs using the following specialization of the standard ABC likelihood approximation \eqref{secABCapproxdef}, proposed in \citep{jassinmarmcc2010}, for when the observations are generated by a HMM.  Specifically, given a sequence of observations $\hat{Y}_{1} , \ldots , \hat{Y}_{n}$ from a HMM, we shall approximate the corresponding likelihood functions with the probabilities
\begin{equation}
\mathbb{P}_{\theta} \left(  Y_{1} \in B^{\epsilon}_{\hat{Y}_{1}} , \ldots, Y_{n} \in
B^{\epsilon}_{\hat{Y}_{n}} \right)  \, \label{introABCllh}
\end{equation}
where for all $y \in \mathbb{R}^{m}$, $B^{\epsilon}_{y}$ denotes the ball of radius $\epsilon$ centered around
the point $y$. The benefit of this approach is that it retains the Markovian
structure of the model. This facilitates both simpler Markov chain Monte Carlo (MCMC) (e.g.~\citep{mckcoodea2009}) and sequential Monte Carlo (SMC) (e.g.~\citep{jassinmarmcc2010}) implementation of the ABC
approximation. Furthermore the resulting approximation has a structure which is particularly tractable to mathematical analysis.

The purpose of this paper is to show that one can prove results about the asymptotic behaviour of ABC based parameter estimators analogous to the standard results in the literature concerning the asymptotic behaviour of  estimators based on the exact value of the likelihood.  In particular we show that one can develop a theoretical justification of  ABC  parameter estimation procedures based on their large sample properties analogous to those provided for Bayesian and maximum likelihood based procedures by the standard Bernstein-von Mises and asymptotic  consistency and normality results respectively.  Our approach   is based on the observation in \citep{deasinjaspet2010} that ABC can be considered as performing parameter estimation using the likelihoods of a collection of perturbed HMMs which suggests that in some sense ABC based parameter estimators should inherit their behaviour from the standard statistical estimators.  We first show that unlike the MLE, which is
asymptotically consistent, the ABC MLE estimator has an innate asymptotic bias in the sense that the value of the estimator converges to the wrong point in parameter space as the number of observations tends to infinity.  Moreover we show that asymptotically the ABC MLE is normally distributed around this biased estimate.  Secondly we show that the resulting ABC Bayesian posterior distributions obey a Bernstein-von Mises type theorem but that the posteriors are again asymptotically biased in the sense that as the number of data points goes to infinity the resulting posterior distributions concentrate about the limit of the ABC MLE rather than the true parameter value.  Finally we show that the size of the asymptotic bias of both the ABC Bayesian and ABC MLE estimators goes to zero as $\epsilon$ tends to zero and under mild regularity conditions we obtain sharp rates for this convergence.   Together these results show that ABC based parameter estimates are asymptotically biased with a bias which can be made arbitrarily small by taking a suitable choice of $\epsilon$ and thus provide a rigorous justification for performing statistical inference based on ABC approximations to the likelihood.

We note that the results in this paper extend those in \citep{deasinjaspet2010} in several ways.  In particular  we provide a  much sharper analysis of the ABC MLE than that contained in \citep{deasinjaspet2010}.  The crucial difference between the current paper and \citep{deasinjaspet2010} is that it is not possible using the techniques of \citep{deasinjaspet2010} to show that the ABC MLE has a unique limit point.  In contrast, in this paper we show that for sufficiently small values of $\epsilon$ the ABC MLE has one and only one limit point.  This then enables us to extend the scope of the analysis in \citep{deasinjaspet2010} to include asymptotic normality results for the ABC MLE and Bernstein-von Mises type results for ABC based Bayesian estimators.

This paper is structured as follows. In Section \ref{sec:notassump} the
notation and assumptions are given and in Section \ref{sec:standABC} we present our main results concerning the asymptotic behaviour of ABC.  The article is summarized in Section \ref{sec:summary} and supporting technical lemmas and proofs of some of the theoretical results are housed
in the four appendices.

\section{Notation and Assumptions}

\label{sec:notassump}

\subsection{Notation and Main Assumptions}

\label{subsecNotandassumptions}

Throughout this paper we shall use lower case letters $x,y,z$ to
denote dummy variables and upper case letters $X,Y,Z$ to denote random
variables. Observations of a random variable, i.e.~data, will be denoted by $\hat{Y}$.
Given any $\epsilon>0$ and $y\in\mathbb{R}^{m}$ we shall let $B_{y}^{\epsilon}$ 
denote the closed ball of radius $\epsilon$ centered on the point
$y$ and let $\mathcal{U}_{B_{y}^{\epsilon}}$ denote the uniform
distribution on $B_{y}^{\epsilon}$. For any $A\subset\mathbb{R}^{m}$ the indicator function of $A$
will be denoted by $\mathbb{I}_{A}$.

In what follows we need to refer to various different scalar, vector and matrix norms.  Given a scalar $z$ and a vector $a$ we shall let $\left\vert z \right\vert$ and $\left\vert a \right\vert$ denote the standard Euclidean scalar and vector norms respectively and   for any matrix $M$ we shall let $\left\Vert M \right\Vert$ denote the Frobenius norm.  We note that although using $\left\vert \cdot \right\vert$ to denote multiple norms is an abuse of notation there is in practice no loss of clarity as the precise meaning of these terms will always be made clear by the context in which they are used.

For any vector of variables $a$ we shall let $\nabla_{a}$ denote the gradiant operator with respect to $a$.  Moreover given vectors of variables $a,b,c$ of dimensions $d_{1}, d_{2}$ and $d_{3}$ we shall let $\nabla_{a} \nabla_{b}$ and $\nabla_{a} \nabla_{b} \nabla_{c}$ denote the $d_{1} \times d_{2}$ and  $d_{1} \times d_{2} \times d_{3}$ matricies of partial derivatives with entries given by $\frac{ \partial^{2} }{ \partial a_{i} b_{j} }$ and $\frac{ \partial^{3} }{ \partial a_{i} b_{j} c_{k} }$ respectively.  Further, for any vector of variables $a$ we shall let $\nabla_{a}^{2}$ and $\nabla_{a}^{3}$ denote $\nabla_{a} \nabla_{a}$ and $\nabla_{a} \nabla_{a} \nabla_{a}$ respectively.  Further given vectors $u,v,w$ we shall let $u \ast v$ and $u \ast v \ast w$ denote the outer products of $u,v$ and $u,v,w$ and $u^{\ast 2}$ and $u^{\ast 3}$ denote the outer products $u \ast u$ and $u \ast u \ast u$ respectively.

It is assumed that for any HMM the hidden state $\left\{ X_{k}\right\} _{k\geq0}$
is time-homogenous and takes values in a compact Polish space $\mathcal{X}$
with associated Borel $\sigma$-field $\mathcal{B}\left(\mathcal{X}\right)$.
Throughout this paper it will be assumed that we  have a collection of HMMs all defined on the same state space and parametrised by some parameter vector $\theta$ taking values in a {\em connected} compact set $\Theta\in\mathbb{R}^{d}$. Furthermore we shall reserve $\theta^{\ast}$ to denote the `true' value of the parameter vector $\theta$.  For
each $\theta\in\Theta$ we shall let $Q_{\theta}\left(x,\cdot\right)$ denote
the transition kernel of the corresponding Markov chain and for each
$x\in\mathcal{X}$ and $\theta\in\Theta$ we assume that $Q_{\theta}\left(x,\cdot\right)$
has a density $q_{\theta}\left(x,\cdot\right)$ w.r.t.~some common
finite dominating measure $\mu$ on $\mathcal{X}$. The initial distribution
of the hidden state will be denoted by $\pi_{0}$.

We also assume that the observations $\left\{ Y_{k}\right\} _{k\geq0}$
take values in a state space $\mathcal{Y}\subset\mathbb{R}^{m}$ for
some $m\geq1$. Furthermore, for each $k$ we assume that the random
variable $Y_{k}$ is conditionally independent of $\ldots,X_{k-1};X_{k+1},\ldots$
and $\ldots,Y_{k-1};Y_{k+1},\ldots$ given $X_{k}$ and that the conditional
laws have densities $g_{\theta}\left(y\vert x\right)$ w.r.t.~some
common $\sigma$-finite dominating measure $\nu$. We further assume
that for every $\theta$ the joint chain $\left\{ X_{k},Y_{k}\right\} _{k\geq0}$
is positive Harris recurrent and has a unique invariant distribution
$\pi_{\theta}$.  For each $\theta \in \Theta$ we shall let $\overline{\mathbb{P}}_{\theta}$
denote the law  of stationary distribution of the corresponding HMM and $\overline{\mathbb{E}}_{\theta}$
denote expectations with respect to the stationary distribution $\overline{\mathbb{P}}_{\theta}$.

We shall frequently have to refer to various kinds of both finite, infinite and doubly infinite sequences.  For brevity the following shorthand notations are used.  For any pair of integers $k \leq n$, $Y_{k:n}$ denotes the sequence of random variables $Y_{k} , \ldots ,Y_{n}$; $Y_{-\infty:k}$ denotes the sequence $\ldots ,Y_{k}$; $Y_{n:\infty}$ denotes the sequence $Y_{n} , \ldots $ and $Y_{-\infty:k;n:\infty}$ denotes the sequence $\ldots, Y_{k} ; Y_{n}, \ldots$.  Further given a measure $\mu$ on a Polish space $\mathcal{X}$ we let $\int \cdot \, \mu(d x_{1:n})$ denote integration w.r.t. the n-fold product measure $\mu^{\otimes n}$ on the n-fold product space $\mathcal{X}^{n}$.  

For any two probability measures $\mu_{1},\mu_{2}$ on a measurable
space $(E,\mathscr{E})$ we let $\|\mu_{1}-\mu_{2}\|_{TV}$ denote the total
variation distance between them. For all $p \in [1,\infty )$ we let
$L_{p}(\mu)$ denote the set of real valued measurable functions satisfying
$\int\left\vert f(x)\right\vert ^{p}\mu(dx)<\infty$.

Finally we note that when writing the likelihood $p_{\theta} ( \hat{Y}_{1} , \ldots , \hat{Y}_{n} )$ of a sequence of observations $ \hat{Y}_{1} , \ldots , \hat{Y}_{n} $ we shall typically suppress the dependence of the likelihood function on the the initial condition
of the hidden state of the process unless we specifically need to refer to it in which case we shall write the likelihood as $p_{\theta} ( \hat{Y}_{1} , \ldots , \hat{Y}_{n} \vert X_{0} = x)$.

\subsection{Particular Assumptions}

\label{subsecParticularAssumptions}

In addition to the assumptions above, the following particular assumptions are
made at various points in the article.

\begin{hypA} 1 The parameter vector $\theta^{\ast}$\ belongs to the interior of
$\Theta$ and $\theta=\theta^{\ast}$ if and only if $\overline{\mathbb{P}%
}_{\theta} ( \ldots , Y_{-1} , Y_{0} , Y_{1} , \ldots )  =\overline{\mathbb{P}}_{\theta^{\ast}} ( \ldots , Y_{-1} , Y_{0} , Y_{1} , \ldots ) $.
\end{hypA}

\begin{hypA} 2 For all $y\in\mathcal{Y}$, $x,x^{\prime}\in\mathcal{X}$,
the mappings $\theta\rightarrow q_{\theta}(x,x^{\prime})$ and $\theta\rightarrow g_{\theta}(\left.y\right\vert x)$
are three times continuously differentiable w.r.t. $\theta$. \end{hypA}

\begin{hypA} 3 There exist constants $\underline{c}_{1},\overline{c}_{1}\in(0,\infty)$
such that for every $y\in\mathcal{Y}$, $x,x^{\prime}\in\mathcal{X}$,
$\theta\in\Theta$ \begin{equation}
\begin{gathered}\underline{c}_{1}\leq q_{\theta}(x,x^{\prime})\leq\overline{c}_{1},\\
g_{\theta}(\left.y\right\vert x)\leq\overline{c}_{1}.\end{gathered}
\label{eq:Cond3a}\end{equation}
 \end{hypA}

\begin{hypA} 4 There exists a constant $\overline{c}_{2}\in(0,\infty)$
such that for every $y\in\mathcal{Y}$, $x,x^{\prime}\in\mathcal{X}$,
$\theta\in\Theta$ \begin{align*}
\left\vert \nabla_{\theta}\log q_{\theta}(x,x^{\prime})\right\vert ,\vert\nabla_{\theta}^{2}\log q_{\theta}(x,x^{\prime})\vert\leq\overline{c}_{2}.\end{align*}
 \end{hypA}

\begin{hypA} 5 For all $\theta \in \Theta$
\begin{equation} 
0 < \int_{\mathcal{X}}  g_{\theta} \left( y \vert x \right) \mu ( dx )  < \infty \label{eq:Cond5a}
\end{equation}
for all $y \in \mathcal{Y}$.
 \end{hypA}

 \begin{hypA} 6
 For any $K>0$
  \begin{equation}
\begin{gathered}E_{\theta^{\ast}}\left[\sup_{\theta \in \Theta} \sup_{x \in \mathcal{X}} \sup_{z\in B_{0}^{K}}\left\Vert \nabla_{\theta}\log g_{\theta}\left(Y +z\vert x\right)\right\Vert ^{3}\right],\\
E_{\theta^{\ast}}\left[\sup_{\theta \in \Theta} \sup_{x \in \mathcal{X}} \sup_{z\in B_{0}^{K}} \left\Vert \nabla_{\theta}^{2}\log g_{\theta}\left(Y + z\vert x\right)\right\Vert^{2} \right] , \\ 
E_{\theta^{\ast}}\left[\sup_{\theta \in \Theta} \sup_{x \in \mathcal{X}} \sup_{z\in B_{0}^{K}} \left\Vert \nabla_{\theta}^{3}\log g_{\theta}\left(Y + z \vert x\right)\right\Vert \right] \leq\infty  . \label{eq:Cond5b}
\end{gathered}
\end{equation}
 \end{hypA}
 
 \begin{remark}
 Assumptions (A1)-(A6) are similar to those used in \citep{doumouryd2004} to prove consistency of the MLE for HMMs.  We use similar assumptions in this paper as, broadly speaking, our approach will be to show that the ABC parameter estimators inherit their properties from standard statistical estimators.  However the methods and emphasis of this paper differ from those in \citep{doumouryd2004} and as a result the assumptions we require have a slightly different flavour.  In particluar we shall require slightly stronger conditions on the differentiability of the conditonal densities $g_{\theta} ( y \vert x)$ but slightly weaker conditions on their integrability.   
 \end{remark}

\begin{remark}
In general assumptions (A3)-(A6) will hold when the state space $\mathcal{X}$ is compact.  However we expect that the behaviours predicted by Theorems \ref{thmABCMLEconsistency}, \ref{thmABCMLEnormality}, \ref{secABCPostBernVMthm} and \ref{secABCsmallepsthmGenCase}  will provide a good qualitative guide to the behaviour of ABC MLE in practice even in cases where the underlying HMMs do not satisfy these assumptions.    
\end{remark}

\section{Approximate Bayesian Computation}

\label{sec:standABC}

\subsection{Structure of ABC Estimators} \label{ABCProcedureApproxDiscussionSubSection}

Suppose that a collection of HMMs
\begin{equation}  \label{origHMM}
\left\{  X_{k} , Y_{k}  \right\}  _{k \geq0}
\end{equation}
parameterised by some $\theta\in\Theta$ are given.  For any sequence of observations $\hat{Y}_{1} , \ldots , \hat{Y}_{n}$ for $\theta\in\Theta$ let $p_{\theta} ( \hat{Y}_{1} , \ldots , \hat{Y}_{n} )$ denote the likelihood of the observations under the corresponding HMM \eqref{origHMM}.  Following \citep{jassinmarmcc2010} we consider approximating $p_{\theta} ( \hat{Y}_{1} , \ldots , \hat{Y}_{n} )$ by the ABC  approximation,
\begin{align}
\lefteqn{ \mathbb{P}_{\theta}\left(  Y_{1} \in B_{\hat{Y}_{1}}^{\epsilon} , \ldots,
Y_{n} \in B_{\hat{Y}_{n}}^{\epsilon} \right) } \notag \\
&&  = \int_{\mathcal{X}^{n+1}\times\mathcal{Y}^{n}} \bigg[\prod_{k=1}^{n}
q_{\theta} (x_{k-1}, x_{k})\mathbb{I}_{B^{\epsilon}_{\hat{Y}_{k}}}(y_{k})
g_{\theta} (y_{k} \vert x_{k}) \bigg] \pi_{0} (d x_{0}) \, \mu( d
x_{1:n} ) \nu( d y_{1:n} ) . \notag \\
&& \label{eq:approxnewver1}
\end{align}
The purpose of this paper is to analyse the asymptotic properties of likelihood based parameter estimators implemented using the ABC approximate likelihoods \eqref{eq:approxnewver1}.  The key to our analysis is the following observation, see \citep{deasinjaspet2010} for more details;
\begin{eqnarray}
\lefteqn{\int_{\mathcal{X}^{n+1}\times\mathcal{Y}^{n}} \bigg[\prod_{k=1}^{n}
q_{\theta} (x_{k-1}, x_{k})\mathbb{I}_{B^{\epsilon}_{\hat{Y}_{k}}}(y_{k})
g_{\theta} (y_{k} \vert x_{k}) \bigg] \pi_{0} (d x_{0}) \, \mu( d
x_{1:n} ) \nu( d y_{1:n} )  \notag } \\
&&  \qquad \qquad   \qquad \qquad \propto\int_{\mathcal{X}^{n+1}} \bigg[\prod_{k=1}^{n} q_{\theta} (x_{k-1},
x_{k}) g^{\epsilon}_{\theta} (\hat{Y}_{k} \vert x_{k})\bigg] \pi_{0} (d x_{0}) \mu(d
x_{1:n})  \label{eq:approx}%
\end{eqnarray}
where
\begin{equation}
\label{EqnPertCondLaw} g^{\epsilon}_{\theta} (y \vert x) = \frac{1}{ \nu \left(
B^{\epsilon}_{y} \right) } \int_{B^{\epsilon}_{y}} g_{\theta} (y^{\prime}
\vert x) \, \nu( d y^{\prime} ) .
\end{equation}

The crucial point is that the quantity $g^{\epsilon}_{\theta} ( y \vert x )$ defined in \eqref{EqnPertCondLaw} is the density of the measure obtained by convolving the measure corresponding to $g_{\theta} ( y \vert x )$ with  $\mathcal{U}_{B^{\epsilon}_{0}}$ where the density is taken w.r.t.~the new dominating measure obtained by convolving $\nu$ with $\mathcal{U}_{B^{\epsilon}_{0}}$.  One can then immediately see that the quantities $q_{\theta} (x, x^{\prime})$ and $g^{\epsilon}_{\theta} (y \vert x)$ appearing in \eqref{eq:approx} are the
transition kernels and conditional laws respectively for a perturbed HMM
$\left\{  X_{k} , Y^{\epsilon}_{k} \right\}  _{k \geq0}$ defined such that
it is equal in law to the process 
\begin{equation}  \label{alg1pertHMM}
\left\{  X_{k} , Y_{k} +
\epsilon Z_{k} \right\}  _{k \geq0}
\end{equation}
where $\left\{  X_{k} , Y_{k} \right\}
_{k \geq0}$ is the original HMM and the $\left\{  Z_{k} \right\}  _{k \geq0}$ are
an i.i.d. sequence of $\mathcal{U}_{B^{1}_{0}}$ distributed random variables.  

\subsection{Theoretical Results} 

It follows  that performing statistical inference using the ABC approximations to the likelihood is equivalent to performing inference using a misspecified collection of models.  It is well known (see for example \citep{whi1982}) that this will in general lead to biased estimates of the true parameter value.   In the rest of this paper we shall investigate the theoretical consequences of this for ABC based parameter estimators.

We start by showing that almost surely the ABC MLE will converge, with increasing sample size, to a given point in parameter space that is not equal to the true parameter value (more generally the set of accumulation points will belong to a given subset of parameter space) and hence that the ABC MLE is asymptotically biased (Theorem \ref{thmABCMLEconsistency}).  Further, we show that these accumulation points must lie in some neighbourhood of
the true parameter value and that the size of this neighbourhood shrinks to zero as $\epsilon$ goes
to zero.  Next we show that for sufficiently small values of $\epsilon$ the ABC MLE has a unique limit point and that asymptotically the ABC MLE is normally distributed about this point with a variance that is proportional to $\frac{1}{n}$ (Theorem \ref{thmABCMLEnormality}).  Third we show that aymptotically the ABC Bayesian posterior converges to that of a Normal random variable, centered on the location of the ABC MLE and with variance again proportional to $\frac{1}{n}$ (Theorem \ref{secABCPostBernVMthm}).  Finally we show that under certain
Lipschitz conditions one can obtain a rate for the decrease in the size of
the asymptotic bias of the ABC parameter estimators (Theorem \ref{secABCsmallepsthmGenCase}).

These results show that the error of ABC based parameter estimators may be decomposed into two parts.  A bias component whose size depends on $\epsilon$ and a variance component whose size is proportional to $\frac{1}{\sqrt{n}}$.  Furthermore they show that the size of the bias can be made arbitrarily small by a suitable choice of $\epsilon$.  Thus taken together the results show that the accuracy of estimators based on ABC approximations to the likelihood can be made to be arbitrarily close to that  of estimators based on the exact value of the likelihood, providing a rigourous mathematical justification for the ABC methodology.

We note that there are two important technical issues that arise in the proofs of these results.  Firstly, as noted in \citep{deasinjaspet2010}, one cannot simply analyse the behaviour of the ABC MLE by extending the parameter space $\Theta$ to include $\epsilon$ and then applying standard results from the theory of MLE because the perturbed likelihoods $g^{\epsilon}_{\theta}(y \vert x)$ are in some sense insufficiently continuous.  Instead one has to establish that in some sense the Lebesgue differentiation theorem still holds upon taking  asymptotic limits.

Secondly we note that   because the dominating measures of the original and perturbed HMMs are no longer necessarily mutually absolutely continuous with respect to each other we can no longer take the standard approach to analysing likelihood based estimators by studying the limits of 
\begin{equation*}
\lim_{n \to \infty} \frac{1}{n} \log p^{\epsilon}_{\theta} ( \hat{Y}_{1},\ldots,\hat{Y}_{n} ) 
\end{equation*}
and interpreting them in terms of Kullback-Leibler distances.  To avoid this problem we instead show that for any $\epsilon$ the relative mean log likelihood surfaces (considered as functions of $\theta$)
\begin{equation*}
\frac{1}{n} \left( \log p^{\epsilon}_{\theta} ( \hat{Y}_{1},\ldots,\hat{Y}_{n} ) - \log p^{\epsilon}_{\theta^{\ast}} ( \hat{Y}_{1},\ldots,\hat{Y}_{n} ) \right)
\end{equation*}
almost surely converge to some limiting surface $l^{\epsilon} ( \theta )$.  The behaviour of ABC based parameter estimators can then be understood by examining the behaviour of the corresponding limiting log likelihood surfaces.  The key result in doing so is the following whose proof is deferred until Appendix B.

\begin{theorem} \label{secABCABCBayesPostthm} Suppose that one has a collection of HMMs parameterized by some parameter
vector $\theta\in\Theta$ that satisfy assumptions (A1)-(A6).  For any $\epsilon \geq 0$ let $p_{\theta}^{\epsilon} ( \cdots) $ denote the likelihood function w.r.t.~the perturbed HMMs \eqref{alg1pertHMM} (and where by definition we let $p_{\theta}^{0} ( \cdots) $ denote the likelihood function of the original HMM \eqref{origHMM}).  Let data $\hat{Y}_{1},\ldots,\hat{Y}_{n}$
generated by the HMM corresponding to an unknown parameter vector
$\theta^{\ast}$ be given.  Then for every $\epsilon \geq 0$ there exists a twice continuously differentiable
function $l^{\epsilon}\left(\theta\right):\Theta\to\mathbb{R}$
such that  for all $x \in \mathcal{X}$   one has that $\bar{\mathbb{P}}_{\theta^{\ast}}$ a.s.
\begin{equation} \label{eq:ABcMLESurf}
\begin{gathered}
\frac{1}{n}\left(\log p_{\theta}^{\epsilon} ( \hat{Y}_{1},\ldots,\hat{Y}_{n} \vert X_{0} = x ) - \log p_{\theta^{\ast}}^{\epsilon} ( \hat{Y}_{1},\ldots,\hat{Y}_{n} )  \vert X_{0} = x \right) \to l^{\epsilon}\left(\theta\right) \\
\frac{1}{n} \nabla_{\theta} \left(\log p_{\theta}^{\epsilon} ( \hat{Y}_{1},\ldots,\hat{Y}_{n}  \vert X_{0} = x ) - \log p_{\theta^{\ast}}^{\epsilon} ( \hat{Y}_{1},\ldots,\hat{Y}_{n}  \vert X_{0} = x  ) \right) \to \nabla_{\theta} l^{\epsilon}\left(\theta\right) \\
\frac{1}{n} \nabla_{\theta}^{2} \left(\log p_{\theta}^{\epsilon} ( \hat{Y}_{1},\ldots,\hat{Y}_{n}   \vert X_{0} = x ) - \log p_{\theta^{\ast}}^{\epsilon} ( \hat{Y}_{1},\ldots,\hat{Y}_{n} )  \vert X_{0} = x \right) \to \nabla_{\theta}^{2} l^{\epsilon}\left(\theta\right)
\end{gathered}
\end{equation}
 uniformly in $\theta$.  
 
 Furthermore $l^{\epsilon}\left(\theta\right), \nabla_{\theta} l^{\epsilon}, \nabla_{\theta}^{2} l^{\epsilon} \to l^{0}\left(\theta\right), \nabla_{\theta} l^{0}, \nabla_{\theta}^{2} l^{0}$ as $\epsilon \to 0$, where the convergence is again uniform in $\theta$.  
\end{theorem}
We can now use Theorem \ref{secABCABCBayesPostthm} to analyse  ABC based parameter estimators by comparing their the asymptotic behaviour (encapsulated in the surfaces $l^{\epsilon} ( \theta )$) to the asymptotic behaviour of estimators based on using the true value of the likelihood (which is encapsulated in the surface $l^{0} ( \theta )$).  we shall start by analysing the behaviour of the ABC MLE which we formally define below.

\begin{procedure}
[ABC MLE]\label{algABCMLE} Given $\epsilon>0$ and data $\hat{Y}_{1},\ldots,\hat{Y}_{n}$, estimate $\theta^{\ast}$ with
\begin{equation} \label{ABCstandestimatorargsupdiffeq999999999999111}
\hat{\theta}^{\epsilon}_{n} = \arg \max_{\theta\in\Theta} \mathbb{P}_{\theta}\left(  Y_{1} \in B_{\hat{Y}_{1}}^{\epsilon} , \ldots,
Y_{n} \in B_{\hat{Y}_{n}}^{\epsilon} \right) .
\end{equation}
\end{procedure}
Using Theorem \ref{secABCABCBayesPostthm} we can now establish the following biased asymptotic consistency and normality type properties of the ABC MLE whose proofs are deferred to Appendix C.

\begin{theorem} \label{thmABCMLEconsistency}
Suppose that one has a collection of HMMs parameterized by some parameter
vector $\theta\in\Theta$ that satisfy assumptions (A1)-(A6).  Let data $\hat{Y}_{1},\ldots,\hat{Y}_{n}$
generated by the HMM corresponding to an unknown parameter vector
$\theta^{\ast}$ be given and suppose that we use the ABC MLE to estimate the value of $\theta^{\ast}$.  Then for every $\epsilon > 0$  there exists a collection of sets $\mathcal{T}^{\epsilon}$ such that  for all initial conditions $X_{0}$ the set of accumulation points of the ABC MLE $\hat{\theta}^{\epsilon}_{n} $ lies $\bar{\mathbb{P}}_{\theta^{\ast}}$ a.s.~in $\mathcal{T}^{\epsilon}$ and 
\begin{equation} \label{ABCMLEmaxset}
\lim_{\epsilon \to 0} \sup_{\theta \in \mathcal{T}^{\epsilon}} \left\vert \theta - \theta^{\ast} \right\vert = 0.
\end{equation}
Furthermore let $l^{0} ( \theta )$ be as in Theorem \ref{secABCABCBayesPostthm}.  If $\nabla_{\theta}^{2} l^{0}\left(\theta^{\ast}\right)$ is strictly negative definite then for sufficiently small values of $\epsilon$ the set $\mathcal{T}^{\epsilon}$ consists of a singleton $\theta^{\ast,\epsilon}$. 
\end{theorem}

\begin{remark} \label{FIremark}
The quantity $-\nabla_{\theta}^{2} l^{0}\left(\theta^{\ast}\right)$ is equal to the asymptotic Fisher information $I$ of the HMM.  For more details see \citep{doumouryd2004}.
\end{remark}

\begin{theorem} \label{thmABCMLEnormality}
Suppose that one has a collection of HMMs parameterized by some parameter
vector $\theta\in\Theta$ that satisfy assumptions (A1)-(A6) and that $\nabla_{\theta}^{2} l^{0}\left(\theta^{\ast}\right)$ is strictly negative definite where $l^{0} ( \theta )$ is as in Theorem \ref{secABCABCBayesPostthm}.  Let data $\hat{Y}_{1},\ldots,\hat{Y}_{n}$
generated by the HMM corresponding to an unknown parameter vector
$\theta^{\ast}$ be given and suppose that we use the ABC MLE to estimate the value of $\theta^{\ast}$.  Then  for sufficiently small values of $\epsilon$ there exists strictly positive definite matricies $J_{\epsilon}, I_{\epsilon}$ such that $\bar{\mathbb{P}}_{\theta^{\ast}}$ a.s. 
\begin{equation}
\sqrt{n}  \left( \hat{\theta}_{n,\epsilon} - \theta^{\ast,\epsilon} \right) \to N (0, I^{-1}_{\epsilon} J_{\epsilon} I^{-1}_{\epsilon} ) .
\end{equation}
Furthermore $J_{\epsilon}, I_{\epsilon} \to I$ as $\epsilon \to 0$ where $I$ is as in Remark \ref{FIremark}.
\end{theorem}
Next we consider the properties of the ABC Bayesian parameter estimator which we define below.

\begin{procedure}
[ABC Bayesian Estimator]\label{algABCBayes} Given $\epsilon>0$ a prior distribution $\pi_{0}$ and data $\hat{Y}_{1},\ldots,\hat{Y}_{n}$ estimate $\theta^{\ast}$ via the ABC posterior
\begin{equation} \label{eq:PostBayesABC}
\pi_{n}^{\epsilon} \propto \mathbb{P}_{\theta}\left(  Y_{1} \in B_{\hat{Y}_{1}}^{\epsilon} , \ldots,
Y_{n} \in B_{\hat{Y}_{n}}^{\epsilon} \right) \pi_{0} .
\end{equation}
\end{procedure}
Given Theorem \ref{secABCABCBayesPostthm} we can easily see that the ABC Bayesian estimator satisfies the following Bernstein-Von Mises type theorem, see \citep{borkalrao1971} whose proof is again deferred to Appendix C.

\begin{theorem} \label{secABCPostBernVMthm} Suppose that the assumptions of Theorem \ref{thmABCMLEnormality} hold and that one  tries to infer the true value of $\theta^{\ast}$ using the ABC approximate Bayesian posterior \eqref{eq:PostBayesABC}.   Suppose further that the prior distribution has a continuous density w.r.t.~Lebesgue measure, then for
sufficiently small values of $\epsilon$ one has that $\bar{\mathbb{P}}_{\theta^{\ast}}$ a.s.
\begin{equation}
\pi^{\epsilon}_{n} \left( \sqrt{n}  ( \theta - \hat{\theta}_{n, \epsilon} )  \right) \to N \left( 0,  I_{\epsilon}^{-1} \right)
\end{equation}
where $I_{\epsilon}$ is as in Theorem \ref{thmABCMLEnormality}.
\end{theorem}

\subsection{Asymptotic Rates of Convergence}

Theorems \ref{thmABCMLEconsistency}, \ref{thmABCMLEnormality} and \ref{secABCPostBernVMthm}
show that asymptotically  ABC based parameter estimators concentrate around a point $\theta^{\ast, \epsilon} \neq \theta^{\ast}$  and thus that the asymptotic bias will be of order $\left|\theta^{\ast,\epsilon}-\theta^{\ast}\right|$.  It is natural to ask at what rate does $\theta^{\ast,\epsilon} \to \theta^{\ast}$ as $\epsilon \to 0$.  We begin our answer to this question with the following example.

\begin{example}
Let $\pi_{1} $ be the distribution on the set of diadic numbers of the form $\frac{1}{4^{k}}; \, k = 0 , 1 , \ldots$ given by $\pi_{1} ( \frac{1}{4^{k}} ) =  \frac{3}{4^{k+1}} $ for all $k$ and let $\pi_{2} $ be the distribution on the set of diadic numbers of the form $\frac{1}{2\cdot4^{k}}$ given by $\pi_{2} ( \frac{1}{2 \cdot 4^{k}} ) =  \frac{3}{4^{k+1}} $ for all $k = 0 , 1 , \ldots$.  Furthermore let $\left\{ \pi_{\theta} \right\}_{\theta \in [0.25,0.75]}$ be the set of distributions defined such that for all $\theta$, $\pi_{\theta} = \theta \pi_{1} + (1- \theta ) \pi_{2}$.  

It is clear that the distributions $\pi_{\theta}$ satisfy the conditions of Theorem  \ref{secABCABCBayesPostthm} and hence that for any $\epsilon$ the limiting approximate mean log likelihood surface $l^{\epsilon}( \theta)$ exists and is well defined. Further if we assume that the true value of the parameter is equal to $\theta^{\ast} = \frac{1}{2}$ then it is easy to show that $\nabla^{2}_{\theta} l^{0} (\theta^{\ast}) \neq 0$ and that for all $k \geq 0$ that $\nabla_{\theta} l^{\frac{1}{4^{k+1}}} ( \theta^{\ast} ) = \frac{3}{4^{k+2}}$ from which it follows that 
\begin{equation*}
\theta^{\ast,\frac{1}{ 4^{k+1}}} -\theta^{\ast}=  \frac{1}{\nabla^{2}_{\theta} l^{0} (\theta^{\ast}) } \frac{3}{4^{k+2}} + o \left( \frac{1}{4^{k+1}} \right) .
\end{equation*}
\end{example}

The above example shows that in the general case one should expect that the size of the asymptotic bias will be at least $O (\epsilon)$.  The next theorem shows that the behaviour of the asymptotic bias will be no worse than this.  In order for it to hold we need to make the following Lipschitz assumptions.

 \begin{hypA} 7
There exists some $R > 0$ such that for all $\epsilon \leq R$.
\begin{equation} \label{eq:CondObsStateDeriv}
\begin{gathered}
E_{\theta^{\ast}}\left[\sup_{x \in \mathcal{X}} \sup_{\theta \in \Theta}  \sup_{z  \in B^{0}_{\epsilon}}  \left\vert \frac{    \nabla_{z} g_{\theta}\left( Y + z \vert x \right) }{ g_{\theta}\left( Y  \vert x \right) }  \right\vert^{2}  \right] , \\
E_{\theta^{\ast}}\left[\sup_{x \in \mathcal{X}} \sup_{\theta \in \Theta}  \sup_{z  \in B^{0}_{\epsilon}}  \left\vert \frac{  \nabla_{z} \left(  \nabla_{\theta}  g_{\theta}\left( Y + z \vert x \right) \right) }{ g_{\theta}\left( Y  \vert x \right) }  \right\vert^{2}  \right]  < \infty .
\end{gathered}
\end{equation}
 \end{hypA}

\begin{theorem} \label{secABCsmallepsthmGenCase}

Suppose that in addition to  all of  the assumptions
of Theorem \ref{secABCPostBernVMthm}  one has that assumption (A7) above also holds.  Then 
\begin{equation}
\left\vert \theta^{\epsilon,\ast}-\theta^{\ast}  \right\vert  = O ( \epsilon ) . \label{eq:ABCBayesLim2ndOrdPert}
\end{equation}
\end{theorem}
Moreover, if the dominating measure $\nu$ is Lebesgue measure then one can show, under slightly stronger Lipschitz assumptions, that the asymptotic error in the ABC parameter estimate is of order $O ( \epsilon )^{2}$.

 \begin{hypA} 8
There exists some $R > 0$ such that for all $\epsilon \leq R$.
\begin{equation} \label{eq:CondObsStateDeriv}
\begin{gathered}
E_{\theta^{\ast}}\left[\sup_{x \in \mathcal{X}} \sup_{\theta \in \Theta}  \sup_{z  \in B^{0}_{\epsilon}}  \left\vert \frac{    \nabla_{z}^{2} g_{\theta}\left( Y + z \vert x \right) }{ g_{\theta}\left( Y  \vert x \right) }  \right\vert^{2}  \right] , \\
E_{\theta^{\ast}}\left[\sup_{x \in \mathcal{X}} \sup_{\theta \in \Theta}  \sup_{z  \in B^{0}_{\epsilon}}  \left\vert \frac{  \nabla_{z}^{2} \left(  \nabla_{\theta}  g_{\theta}\left( Y + z \vert x \right) \right) }{ g_{\theta}\left( Y  \vert x \right) }  \right\vert^{2}  \right]  < \infty .
\end{gathered}
\end{equation}
 \end{hypA}

\begin{theorem} \label{secABCsmallepsthmLeb}

Suppose that $\nu$ is Lebesgue measure and that in addition to  all of  the assumptions
of Theorem \ref{secABCsmallepsthmGenCase}  one has that assumption (A8) above holds also.  Then 
\begin{equation}
\left\vert \theta^{\epsilon,\ast}-\theta^{\ast}  \right\vert  = O ( \epsilon^{2} ) . \label{eq:ABCBayesLebLim2ndOrdPert}
\end{equation}
\end{theorem}
The proofs of Theorems \ref{secABCsmallepsthmGenCase} and \ref{secABCsmallepsthmLeb} are deferred to Appendix D.  Finally we note that in the case that $\nu$ is Lebesgue measure we have from Theorems \ref{thmABCMLEnormality} and \ref{secABCPostBernVMthm} that the variance of ABC based based estimators is of order $O ( 1/ \sqrt{n} )$ while their bias is of order $O ( \epsilon^{2} )$.  It follows that (at least in theory) it is optimal to scale $\epsilon$ as $O ( 1 / \sqrt[4]{n} )$ as $n$ goes to infinity.  Intriguingly this is the same rate as the optimal bandwidth in kernel density estimation (see for example \citep{wanjon1995}).  This suggests an alternative interpretation of ABC as approximating the likelihood via a kind of kernel density based estimate.

\section{Summary} \label{sec:summary}

In this paper we have shown that the framework developed in \citep{deasinjaspet2010} to analyse the behaviour of the the  ABC MLE can be extended to provide a rigourous analysis  of the behaviour of ABC based estimators in both the Bayesian and frequentist contexts.  In particular we have shown that ABC based parameter estimators satisfy results analogous to the asymptotic consistency, asymptotic normality and Bernstein-von Mises theorems for standard parameter estimators but that the ABC estimators are asymptotically biased.  Furthermore we have shown that this asymptotic bias can be made arbitrarily small by choosing a sufficiently small value of the parameter $\epsilon$.   Together these theoretical resultshelp to solidify and extend existing intuition and provide a rigourous theoretical justification for ABC based parameter estimation procedures.

\renewcommand{\thesection}{A}
\renewcommand{\theequation}{A-\arabic{equation}}

\section*{Appendix A: Auxillary Results}

In this section we present without proof some well known results that will be needed in the proofs of Theorems \ref{secABCABCBayesPostthm}, \ref{thmABCMLEconsistency}, \ref{thmABCMLEnormality}, \ref{secABCPostBernVMthm} and \ref{secABCsmallepsthmGenCase}.  The first two lemmas are standard
result from real analysis.

\begin{lemma} \label{lemrealanalresstand2} 
Let a connected compact set $G \subset \mathbb{R}^{u}$ and some constant $K>0$ be given.  Suppose that there exists
a continuous function $f: G \to\mathbb{R}^{v}$ and sequence
of continuous  functions $f_{n}: G \to\mathbb{R}^{v}$,
$n\geq1$, such that for all $n$ the function $f_{n}$ is Lipschitz-$K$ continuous. Then $f_{n} \to f$ uniformly in $G$ if and only if $f_{n} \to f$ pointwise on a countable dense subset of  $G$.
\end{lemma}

\begin{lemma} \label{lemrealanalresstand} Let a connected compact set $G \subset \mathbb{R}^{u}$ be given and suppose that there exists
a continuous function $g: G \to\mathbb{R}^{v}$ and sequence
of continuously differentiable functions $f_{n}: G \to\mathbb{R}^{v}$,
$n\geq1$, such that $\nabla f_{n}(z) \to g(z)$  uniformly in $z$ and $f_{n}(z^{\ast})$ is Cauchy for some $z^{\ast} \in G$.
Then there exists a uniformly bounded and  continuously differentiable function $f$ such that $ f_{n}(z) \to f(z)$ 
uniformly in $z$ and $\nabla f(z) = g(z)$.
\end{lemma}

Lemmas \ref{doumourydProp4}, \ref{doumourydProp41} and \ref{doumourydProp42} are essentially corollaries and extensions of Propositions 4 and 5 in \citep{doumouryd2004} and may be proved in exactly the same manner.  We leave the details to the reader.

\begin{lemma} \label{doumourydProp4}

Suppose that one has a collection of HMMs parameterised some vectors
$\theta\in\Theta$ that satisfy assumption (A2).  Furthermore suppose that one has a HMM $\left\{ X_{k}, Y_{k}\right\} _{k\geq1}$, defined on the same state spaces as the parameterised collection of HMMs, which satisfies assumption (A2) with the same values of $\underline{c}$ and $\overline{c}$.   

Given measurable functions $\phi_{1}, \phi_{2}, \phi_{3}: \Theta \times\mathcal{X}^{2} \times\mathcal{Y}\to\mathbb{R}$
 and  $y \in \mathcal{Y}$, $k<l$ and $s \in \left\{ 1,2,3 \right\} $ define the following functions of the HMM $\left\{ X_{k}, Y_{k}\right\} _{k\geq1}$
\begin{align*}
\phi_{s;k:l}(\theta) \triangleq \sum_{i=k+1}^{l} \phi_{s} \left(\theta, X_{i-1},X_{i},Y_{i} \right) 
\end{align*}
and for  any $n>0$ define the random variables $\Delta_{0,n}$, $\Gamma_{0,n}$, $\Psi_{0,n}$ and $\Omega_{0,n}$ by
\begin{align*}
\Delta_{0,n} (\theta  ) \triangleq E_{\theta} \Big[ \phi_{1;-n:0}(\theta) \big\vert Y_{-n:0} \Big] - E_{\theta} \Big[ \phi_{1;-n:-1}(\theta) \big\vert Y_{-n:-1} \Big],
\end{align*}
\begin{align*}
& \Gamma_{0,n} (\theta ) \triangleq E_{\theta} \Big[ \phi_{1;-n:0}(\theta) \phi_{2;-n:0}(\theta)\big\vert Y_{0:-n} \Big] - E_{\theta} \Big[  \phi_{1;-n:-1}(\theta) \phi_{2;-n:-1}(\theta) \big\vert Y_{-n:-1} \Big]  \\
& \qquad \qquad  + E_{\theta} \Big[ \phi_{1;-n:-1}(\theta) \big\vert Y_{-n:-1} \Big]  E_{\theta} \Big[ \phi_{2;-n:-1}(\theta) \big\vert Y_{-n:-1} \Big] \\
& \qquad \qquad \qquad \qquad \qquad - E_{\theta} \Big[ \phi_{1;-n:0}(\theta) \big\vert Y_{-n:-0} \Big]  E_{\theta} \Big[ \phi_{2;-n:0}(\theta) \big\vert Y_{-n:-0} \Big]  ,
\end{align*}
\begin{align*}
& \Psi_{0,n} (\theta ) \triangleq E_{\theta} \Big[ \phi_{1;-n:0}(\theta) \phi_{2;-n:0}(\theta)\big\vert Y_{0:-n} \Big] E_{\theta} \Big[ \phi_{3;-n:0}(\theta) \big\vert Y_{-n:-0} \Big]  \\
& \qquad - E_{\theta} \Big[ \phi_{1;-n:0}(\theta) \big\vert Y_{-n:-0} \Big]  E_{\theta} \Big[ \phi_{2;-n:0}(\theta) \big\vert Y_{-n:-0} \Big]  E_{\theta} \Big[ \phi_{3;-n:0}(\theta) \big\vert Y_{-n:-0} \Big]   \\
& \quad + E_{\theta} \Big[ \phi_{1;-n:-1}(\theta) \big\vert Y_{-n:-1} \Big]  E_{\theta} \Big[ \phi_{2;-n:-1}(\theta) \big\vert Y_{-n:-1} \Big] E_{\theta} \Big[ \phi_{3;-n:-1}(\theta) \big\vert Y_{-n:-1} \Big]  \\
&  \qquad \quad - E_{\theta} \Big[  \phi_{1;-n:-1}(\theta) \phi_{2;-n:-1}(\theta) \big\vert Y_{-n:-1} \Big] E_{\theta} \Big[ \phi_{3;-n:-1}(\theta) \big\vert Y_{-n:-1} \Big]  ,
\end{align*}
and 
\begin{align*}
& \Omega_{0,n} (\theta) \triangleq E_{\theta} \Big[ \phi_{1;-n:0}(\theta) \phi_{2;-n:0}(\theta) \phi_{3;-n:0}(\theta) \big\vert Y_{-n:-0} \Big]  \\
& \qquad - E_{\theta} \Big[ \phi_{1;-n:0}(\theta) \big\vert Y_{-n:-0} \Big]  E_{\theta} \Big[ \phi_{2;-n:0}(\theta) \big\vert Y_{-n:-0} \Big]  E_{\theta} \Big[ \phi_{3;-n:0}(\theta) \big\vert Y_{-n:-0} \Big]   \\
& \quad + E_{\theta} \Big[ \phi_{1;-n:-1}(\theta) \big\vert Y_{-n:-1} \Big]  E_{\theta} \Big[ \phi_{2;-n:-1}(\theta) \big\vert Y_{-n:-1} \Big] E_{\theta} \Big[ \phi_{3;-n:-1}(\theta) \big\vert Y_{-n:-1} \Big]  \\
&  \qquad \quad - E_{\theta} \Big[  \phi_{1;-n:-1}(\theta) \phi_{2;-n:-1}(\theta) \phi_{3;-n:-1}(\theta) \big\vert Y_{-n:-1} \Big]  .
\end{align*}
Then there exist $\sigma (Y_{-\infty : 0 }) $ measurable random variables $\Delta_{0,\infty}( \theta  )$, $\Gamma_{0,\infty}( \theta  )$, $\Psi_{0,\infty}( \theta  )$ and $\Omega_{0,\infty}( \theta  )$ and constants $C < \infty$ and $0 < \rho < 1$ which depend only on $\underline{c}$ and $\overline{c}$ such that for any initial condition on the collection of parameterised HMMs
\begin{equation} \label{eq:LemdoumourydPropBounds6}
\begin{gathered}
 \bar{E}  \left[ \sup_{ \theta \in \Theta } \left| \Delta_{0,n} ( \theta  ) - \Delta_{0,\infty} ( \theta  )  \right|\right] \leq C \rho^{n} \bar{E}  \Big[ \left\Vert \phi_{1} \right\Vert_{\infty} \Big]   \\
 \bar{E}  \left[ \sup_{\theta \in \Theta} \left|\Gamma_{0,n} ( \theta  ) - \Gamma_{0,\infty} ( \theta   )   \right|\right ] \leq C \rho^{n}  \sup_{s \in \left\{ 1 , 2 \right\} } \bar{E}  \Big[ \left\Vert \phi_{s} \right\Vert_{\infty}^{2}  \Big]   \\
 \bar{E}  \left[ \sup_{\theta \in \Theta} \left|\Psi_{0,n} ( \theta  ) - \Psi_{0,\infty} ( \theta  )   \right|\right] \leq C \rho^{n}  \sup_{s \in \left\{ 1 , 2 , 3 \right\} }  \bar{E}   \Big[ \left\Vert \phi_{s} \right\Vert_{\infty}^{3}  \Big]   \\
 \bar{E}  \left[ \sup_{\theta \in \Theta} \left|\Omega_{0,n} ( \theta  ) - \Omega_{0,\infty} ( \theta  )   \right|\right]  \leq C \rho^{n} \sup_{s \in \left\{ 1 , 2 , 3 \right\} } \bar{E}   \Big[ \left\Vert \phi_{s} \right\Vert_{\infty}^{3}  \Big]
\end{gathered}
\end{equation}
where  for all   $s \in \left\{ 1,2,3 \right\} $
\begin{align*}
\left\Vert \phi_{s} \right\Vert_{\infty} (y) \triangleq \sup_{ \theta \in \Theta}\sup_{x, x^{\prime} \in \mathcal{X}}\left|\phi_{s} \left(\theta,x,x^{\prime} ,y \right)\right|  
\end{align*}
$\bar{E} \left[ \cdot \right]$ denotes expectation w.r.t.~the law and stationary law respectively of the process $\left\{ X_{k}, Y_{k}\right\} _{k\geq1}$.
 \end{lemma}
 
 \begin{lemma}  \label{doumourydProp41}
 Suppose that the assumptions of Lemma \ref{doumourydProp4} all hold.  Then there exist constants $C < \infty$ and $0 < \rho < 1$ such that for any initial condition on the collection of parameterised HMMs
 \begin{align}
  \bar{E}  \left[ \sup_{ \theta \in \Theta } \left| \Delta_{0,n} ( \theta  ) - \Delta_{0,\infty} ( \theta  )  \right|^{2} \right] \leq C \rho^{n} \bar{E}   \Big[ \left\Vert \phi_{1} \right\Vert_{\infty}^{2}  \Big]  .   \label{doumourydProp41eq}
 \end{align}
 \end{lemma}
 
  \begin{lemma}  \label{doumourydProp42}
 Let the same assumptions and notation as Lemma \ref{doumourydProp4} be given.  Then  there exist constants $C < \infty$ and $0 < \rho < 1$  such that for any $k,n$
 \begin{align}
  \bar{E}  \bigg[ \Big\vert  \bar{E} \left[  \Delta_{0,n} ( \theta  ) \vert Y_{-\infty:-k} \right]  -  \bar{E} \left[  \Delta_{0,n} ( \theta  ) \vert Y_{-\infty:-k-1} \right]  \Big\vert^{2} \bigg]   \leq C \rho^{k}   \bar{E}    \Big[ \left\Vert \phi_{1} \right\Vert_{\infty}^{2}  \Big]    \label{doumourydProp42eq}
 \end{align}
 where $ \bar{E} \left[ \cdot \vert \cdot \right]$ denotes conditional expectation w.r.t.~the law of the process $\left\{ X_{k}, Y_{k}\right\} _{k\geq1}$.
 \end{lemma}

The last Lemma is a statement of the Fisher identity  and the
Louis missing information  principle  (see for example \citep{doumouryd2004})  plus an extension of these results to third order derivatives of the log likelihood function.  Given assumptions (A2)-(A6) it follows from a simple application of the dominated convergence theorem.

\begin{lemma} \label{LemFishInf}
Suppose that assumptions (A2)-(A6) hold for a collection of HMMs parametrised by some vector $\theta \in \Theta$ where for each $\theta \in \Theta$ we let $g_{\theta} \left( y \vert x \right)$ and $q_{\theta}\left(x^{\prime},x \right)$ denote the densities of the conditional law and transition kernel of the corresponding HMM.  For any $\epsilon \geq 0$ let $g_{\theta}^{\epsilon} \left( y \vert x \right)$ denote the density of the conditional law of the corresponding perturbed HMM \eqref{alg1pertHMM}.  By convention we let $g_{\theta}^{0} \left( y \vert x \right) = g_{\theta} \left( y \vert x \right)$.

For any $\theta \in \Theta$, $\epsilon \geq 0$ and $n > 0$ let $\psi (\theta, x, x^{\prime}, y) = \log g_{\theta}^{\epsilon}\left( y \vert x^{\prime} \right)q_{\theta}\left(x,x^{\prime} \right)$ and following the notation of Lemma \ref{doumourydProp4}  let $\psi_{n} (\theta) \triangleq \sum_{i=1}^{n} \psi (\theta, X_{i-1}, X_{i}, Y_{i})$.  Then one has that for any $\theta \in \Theta$ and $\epsilon \geq 0$ the log ABC approximate likelihood function $\log p_{\theta}^{\epsilon} ( \cdots ) $  is three times differentiable and 
\begin{equation}
\nabla_{\theta} \log p_{\theta}^{\epsilon}( Y_{1},\ldots, Y_{n}) = E_{\theta^{\epsilon}}\left[\nabla_{\theta} \psi_{n} (\theta) \big\vert  Y_{1:n}\right],\label{LemFishInf1}
\end{equation}
 \begin{align}
& \nabla_{\theta}^{2}\frac{1}{n}\log p_{\theta}^{\epsilon}( Y_{1},\ldots, Y_{n}) \label{LemFishInf2} \\
& \qquad  = E_{\theta^{\epsilon}}\left[\nabla_{\theta}^{2} \psi_{n} \big\vert  Y_{1:n}\right]  + E_{\theta^{\epsilon}}\left[ \left( \nabla_{\theta} \psi_{n} \right)^{\ast 2}  \big\vert  Y_{1:n}\right]  - E_{\theta^{\epsilon}}\left[\nabla_{\theta} \psi_{n} \big\vert  Y_{1:n}\right]^{\ast 2} \notag ,
 \end{align}
 and 
 \begin{align}
& \nabla_{\theta}^{3}\frac{1}{n}\log p_{\theta}^{\epsilon}( Y_{1},\ldots, Y_{n}) =  E_{\theta^{\epsilon}}\left[\nabla_{\theta}^{3} \psi_{n} \big\vert  Y_{1:n}\right]  \notag \\
& + 3 E_{\theta^{\epsilon}}\left[\nabla_{\theta}^{2} \ast \psi_{n} \nabla_{\theta} \psi_{n}  \big\vert  Y_{1:n}\right]  - 3 E_{\theta^{\epsilon}}\left[\nabla_{\theta}^{2} \psi_{n} \big\vert  Y_{1:n}\right]  \ast E_{\theta^{\epsilon}}\left[  \nabla_{\theta} \psi_{n}  \big\vert  Y_{1:n}\right]  \notag \\
& -  3 E_{\theta^{\epsilon}}\left[\left( \nabla_{\theta} \psi_{n} \right)^{\ast 2} \big\vert  Y_{1:n}\right]  \ast E_{\theta^{\epsilon}}\left[  \nabla_{\theta} \psi_{n}  \big\vert  Y_{1:n}\right]   +  E_{\theta^{\epsilon}}\left[ \left( \nabla_{\theta} \psi_{n} \right)^{\ast 3} \big\vert  Y_{1:n}\right]  \notag \\ 
& + 2 E_{\theta^{\epsilon}}\left[\nabla_{\theta} \psi_{n} \big\vert Y_{1:n}\right]^{\ast 3} \label{LemFishInf3}
 \end{align}
 where $E_{\theta^{\epsilon}}\left[\cdot\vert\cdot\right]$ denotes
conditional expectation w.r.t.~the law of the perturbed HMM \eqref{alg1pertHMM}.
\end{lemma}

\renewcommand{\thesection}{B}
\renewcommand{\theequation}{B-\arabic{equation}}

\section*{Appendix B: Proof of Theorem \ref{secABCABCBayesPostthm}} 

Theorem \ref{secABCABCBayesPostthm} is an immediate corollary of  the following three lemmas.

\begin{lemma}\label{AppAABCLLHLimLem}Suppose that assumptions (A1)-(A6) hold for a collection of HMMs parametrised by some vector $\theta \in \Theta$. Then for any $\epsilon \geq 0$ there exists a twice continuously
differentiable function $l^{\epsilon}\left(\theta\right)$ such
that 
\begin{equation} \label{eq:ABCLLHLimLem1}
\begin{gathered}
\lim_{n \to \infty} \sup_{\theta \in \Theta} \left\vert \frac{1}{n}\left(\log p_{\theta}^{\epsilon}( Y_{1},\ldots, Y_{n})-\log p_{\theta^{\ast}}^{\epsilon}( Y_{1},\ldots, Y_{n})\right) - l^{\epsilon}\left(\theta\right) \right\vert = 0 \\
\lim_{n \to \infty} \sup_{\theta \in \Theta} \left\vert \frac{1}{n} \nabla_{\theta}  \left(\log p_{\theta}^{\epsilon}( Y_{1},\ldots, Y_{n})-\log p_{\theta^{\ast}}^{\epsilon}( Y_{1},\ldots, Y_{n})\right) - \nabla_{\theta}  l^{\epsilon}\left(\theta\right) \right\vert = 0 \\
\lim_{n \to \infty} \sup_{\theta \in \Theta} \left\vert \frac{1}{n}  \nabla_{\theta}^{2} \left(\log p_{\theta}^{\epsilon}( Y_{1},\ldots, Y_{n})-\log p_{\theta^{\ast}}^{\epsilon}( Y_{1},\ldots, Y_{n})\right) -  \nabla_{\theta}^{2} l^{\epsilon}\left(\theta\right) \right\vert = 0
\end{gathered}
\end{equation}
$\bar{\mathbb{P}}_{\theta^{\ast}}$ a.s.~and in $L_{1}\left( \bar{\mathbb{P}}_{\theta^{\ast}}\right)$  where for all $\theta$ and $\epsilon$, $p_{\theta}^{\epsilon}( \cdots )$ denotes the  likelihood function of the perturbed HMM \eqref{alg1pertHMM}.  By convention we define $p_{\theta}^{0}( \cdots )$ to be equal to the true likelihood function $p_{\theta}( \cdots )$.
Moreover there exists some constant $0<K<\infty$ such that for all $\theta \in \Theta$ and $\epsilon \geq 0$ 
\begin{equation} \label{eq:ABCLLHLimLem3}
l^{\epsilon}\left(\theta\right) , \nabla_{\theta} l^{\epsilon}\left(\theta\right) , \nabla_{\theta}^{2} l^{\epsilon}\left(\theta\right)  \leq K
\end{equation}
and $l^{\epsilon}\left(\theta\right) , \nabla_{\theta} l^{\epsilon}\left(\theta\right) , \nabla_{\theta}^{2} l^{\epsilon}\left(\theta\right) $ are K-Lipschitz (as functions of $\theta$).
\end{lemma}

\begin{lemma} \label{AppAABCLLHLim1Deriv}
Suppose that assumptions (A1)-(A6) hold for a collection of HMMs parametrised by some vector $\theta \in \Theta$ and for any $\epsilon>0$ let $l^{\epsilon}\left(\theta\right)$ be equal to the corresponding limit function defined in Lemma \ref{AppAABCLLHLimLem}. Then for all $\theta \in \Theta$ one has that
\[
\lim_{\epsilon\to0}\nabla_{\theta} l^{\epsilon}\left(\theta \right)=\nabla_{\theta} l^{0} \left(\theta \right) .
\]

\end{lemma}

\begin{lemma}\label{AppAABCLLHLim2Deriv}Suppose that assumptions (A1)-(A6) hold for a collection of HMMs parametrised by some vector $\theta \in \Theta$ and for any $\epsilon>0$ let $l^{\epsilon}\left(\theta\right)$ be equal to the corresponding limit function defined in Lemma \ref{AppAABCLLHLimLem}. Then for all $\theta \in \Theta$ one has that
 \[
\lim_{\epsilon\to0}\nabla_{\theta}^{2} l^{\epsilon}\left(\theta \right)=\nabla_{\theta}^{2}l^{0} \left(\theta \right).\]
\end{lemma}

In order to complete this section we need to provide the proofs of Lemmas \ref{AppAABCLLHLimLem}, \ref{AppAABCLLHLim1Deriv} and \ref{AppAABCLLHLim2Deriv}.  We start by stating some properties of the perturbed conditional likelihood \eqref{EqnPertCondLaw} that will be needed in the sequel.  First note that it follows from assumptions (A2) and (A5) and a simple application of the dominated convergence theorem that
\begin{equation} \label{AppAABCLLHLim1DerivPf101}
\nabla_{\theta} g^{\epsilon}_{\theta} \left(y\vert x\right)\triangleq\frac{\int_{B_{y}^{\epsilon}} \nabla_{\theta} g_{\theta} \left(z\vert x\right) \nu ( dz )}{\int_{B_{y}^{\epsilon}} \nu\left(dz\right)} 
\end{equation}
and that $\nabla_{\theta} g^{\epsilon}_{\theta} \left(y\vert x\right)$ is continuous w.r.t.~$\theta$ for all $\epsilon$, $x$ and $y$.
Furthermore since 
\begin{equation*}
\int_{B_{y}^{\epsilon}} \nabla_{\theta} g_{\theta} \left(z\vert x\right) \nu ( dz ) \leq \sup_{\theta \in \Theta} \sup_{x \in \mathcal{X}} \sup_{z \in B_{y}^{\epsilon}}  \left(  \frac{\nabla_{\theta} g_{\theta} \left(z\vert x\right)}{g_{\theta} \left(z\vert x\right)} \right) \times  \int_{B_{y}^{\epsilon}} g_{\theta} \left(z\vert x\right) \nu\left(dz\right) 
\end{equation*}
it follows from \eqref{AppAABCLLHLim1DerivPf101} and assumption (A5) that for any $\epsilon > 0$
  \begin{equation} \label{AppAABCLLHLim1DerivPf102}
E_{\theta^{\ast}}\left[\sup_{\theta \in \Theta} \sup_{x \in \mathcal{X}} \left\Vert \nabla_{\theta}\log g_{\theta}^{\epsilon} \left(Y \vert x\right)\right\Vert  \right] < \infty .
\end{equation}
Finally we note that analogous comments hold for $ g^{\epsilon}_{\theta} \left(y\vert x\right)$ and $\nabla^{2}_{\theta} g^{\epsilon}_{\theta} \left(y\vert x\right)$.

We now proceed to the proof of Lemma \ref{AppAABCLLHLimLem}.

\bigskip

\begin{proof}[Proof of Lemma \ref{AppAABCLLHLimLem}]
First note that for any $n$ the gradient of the mean log ABC likelihood
may be decomposed into the following telescoping sum
\begin{equation}
\nabla_{\theta}\frac{1}{n}\log p_{\theta}^{\epsilon}( Y_{1},\ldots, Y_{n})=\frac{1}{n}\sum_{i=1}^{n} h^{\epsilon}_{\theta} (Y_{1:i}) . \label{eq:ABCLLHLimLemPf3}
\end{equation}
where for any $k < n$
\begin{equation} \label{eq:ABCLLHLimLemPf300}
h^{\epsilon}_{\theta} (Y_{k:n}) := \nabla_{\theta}\log p_{\theta}^{\epsilon}( Y_{k},\ldots, Y_{n})-\nabla_{\theta}\log p_{\theta}^{\epsilon}( Y_{k},\ldots, Y_{n-1}) .
\end{equation}
It then follows from  \eqref{eq:LemdoumourydPropBounds6}  and  \eqref{LemFishInf1}    that there exist constants $K < \infty$ and $0 < \rho < 1$ such that for all $\theta \in \Theta$, $\epsilon \geq 0$ and $n > 0$ there exists some $\sigma (Y_{-\infty : 0 }) $ measurable random variable $R^{\epsilon}_{\theta}(Y_{-\infty:0})$ such that 
\begin{equation} \label{AppAABCLLHLimLempf100}
\bar{E}_{\theta^{\ast}}\left[ \sup_{k \geq n} \Big\vert h^{\epsilon}_{\theta} (Y_{-n:0}) - R^{\epsilon}_{\theta}(Y_{-\infty:0})\Big\vert \right] \leq K \rho^{n} .
\end{equation}
We note that by \eqref{AppAABCLLHLim1DerivPf101} and \eqref{AppAABCLLHLim1DerivPf102} and the accompanying comments and the dominated convergence theorem  that $E_{\theta^{\ast}}\left[  h^{\epsilon}_{\theta} (Y_{-n:0})  \right] $ is continuous for all $n$ and hence by \eqref{AppAABCLLHLimLempf100} that $E_{\theta^{\ast}}\left[  R^{\epsilon}_{\theta}(Y_{-\infty:0})  \right] $ is continuous.  Further it then follows from \eqref{AppAABCLLHLimLempf100} and two applications of the ergodic theorem that for any $m >0$
\begin{align}
& \limsup_{n \to \infty} \left\vert \frac{1}{n}\sum_{i=1}^{n} h^{\epsilon}_{\theta} (Y_{i:i}) - E_{\theta^{\ast}}\big[  R^{\epsilon}_{\theta} (Y_{-\infty:0}) \big] \right\vert \notag \\
& \leq \limsup_{n \to \infty} \left\vert \frac{1}{n}\sum_{i=1}^{m} h^{\epsilon}_{\theta} (Y_{1:i}) - E_{\theta^{\ast}}\big[ R^{\epsilon}_{\theta} (Y_{-\infty:0}) \big] \right\vert \notag \\
& + \limsup_{n \to \infty} \left\vert \frac{1}{n}\sum_{i=m+1}^{n}  R^{\epsilon}_{\theta}(Y_{-\infty:i}) - E_{\theta^{\ast}}\big[  R^{\epsilon}_{\theta} (Y_{-\infty:0}) \big] \right\vert \notag \\
& + \limsup_{n \to \infty} \frac{1}{n}\sum_{i=m+1}^{n}\sup_{k \geq 0} \bigg\vert h^{\epsilon}_{\theta} (Y_{1-k:i}) - R^{\epsilon}_{\theta}(Y_{-\infty:i})\bigg\vert \notag\\ 
& \leq K \rho^{m} \label{eq:ABCLLHLimLemPf4}
\end{align}
and
\begin{align}
& \limsup_{n \to \infty} \sup_{\theta \in \Theta}  \left\vert \frac{1}{n} \sum_{i=1}^{n} h^{\epsilon}_{\theta} (Y_{1:i}) \right\vert \leq \limsup_{n \to \infty} \frac{1}{n}\sum_{i=1}^{n} \sup_{\theta \in \Theta} \sup _{k \leq i-1} \bigg\vert  h^{\epsilon}_{\theta} (Y_{k:i}) \bigg\vert \leq K . \label{eq:ABCLLHLimLemPf9}
\end{align}
Thus we have by \eqref{eq:ABCLLHLimLemPf3}, \eqref{eq:ABCLLHLimLemPf4} and \eqref{eq:ABCLLHLimLemPf9} that $\overline{\mathbb{P}}_{\theta^{\ast}}$ a.s.
\begin{equation} \label{eq:ABCLLHLimLemPf11}
\nabla_{\theta}\frac{1}{n}\log p_{\theta}^{\epsilon}( Y_{1},\ldots, Y_{n}) \to E_{\theta^{\ast}}\big[ R^{\epsilon}_{\theta} (Y_{-\infty:0}) \big] 
\end{equation}
pointwise in $\theta$ for some continuous in $\theta$ function  $\bar{E}_{\theta^{\ast}}\big[ R^{\epsilon}_{\theta} (Y_{-\infty:0}) \big] $ and that $\vert \nabla_{\theta}\frac{1}{n}\log p_{\theta}^{\epsilon}( Y_{1},\ldots, Y_{n}) \vert$ is eventually uniformly bounded above by $K$.

Moreover it follows from \eqref{eq:LemdoumourydPropBounds6}, \eqref{LemFishInf2} and \eqref{LemFishInf3} and a similar argument as above that for any $\theta \in \Theta$ and $\epsilon \geq 0$ there exist $\sigma (Y_{-\infty : 0}) $ measurable random variables $S^{\epsilon}_{\theta}(Y_{-\infty:0})$ and $T^{\epsilon}_{\theta}(Y_{-\infty:0})$ such that  $E_{\theta^{\ast}}\left[  S^{\epsilon}_{\theta}(Y_{-\infty:0})  \right] $ and $E_{\theta^{\ast}}\left[  T^{\epsilon}_{\theta}(Y_{-\infty:0})  \right] $ are continuous functions of $\theta$, that
\begin{equation} \label{eq:ABCLLHLimLemPf6}
\begin{gathered}
\nabla_{\theta}^{2} \frac{1}{n}\log p_{\theta}^{\epsilon}( Y_{1},\ldots, Y_{n}) \to \bar{E}_{\theta^{\ast}}\left[ S^{\epsilon}_{\theta}(Y_{-\infty:0}) \right] , \\
\nabla_{\theta}^{3}\frac{1}{n}\log p_{\theta}^{\epsilon}( Y_{1},\ldots, Y_{n}) \to \bar{E}_{\theta^{\ast}} \left[ T^{\epsilon}_{\theta}(Y_{-\infty:0}) \right]
\end{gathered}
\end{equation} 
$\overline{\mathbb{P}}_{\theta^{\ast}}$ a.s.~and in $L^{1}\left(\overline{\mathbb{P}}_{\theta^{\ast}}\right)$  and that $\overline{\mathbb{P}}_{\theta^{\ast}}$ a.s.~eventually $\vert \nabla_{\theta}^{2}\frac{1}{n}\log p_{\theta}^{\epsilon}( Y_{1},\ldots, Y_{n}) \vert$ and $\vert \nabla_{\theta}^{3}\frac{1}{n}\log p_{\theta}^{\epsilon}( Y_{1},\ldots, Y_{n}) \vert$ are both uniformly bounded above by $K$.  Since the fact that $\vert \nabla_{\theta}^{2}\frac{1}{n}\log p_{\theta}^{\epsilon}( Y_{1},\ldots, Y_{n}) \vert$ and $\vert \nabla_{\theta}^{3}\frac{1}{n}\log p_{\theta}^{\epsilon}( Y_{1},\ldots, Y_{n}) \vert$  are both uniformly bounded above implies that both $\vert \nabla_{\theta} \frac{1}{n}\log p_{\theta}^{\epsilon}( Y_{1},\ldots, Y_{n}) \vert$ and $\vert \nabla_{\theta}^{2}\frac{1}{n}\log p_{\theta}^{\epsilon}( Y_{1},\ldots, Y_{n}) \vert$ are Lipschitz the result now follows from Lemmas \ref{lemrealanalresstand2} and \ref{lemrealanalresstand}.
\end{proof}

\bigskip

In remains to prove  Lemmas \ref{AppAABCLLHLim1Deriv} and \ref{AppAABCLLHLim2Deriv}.  Since the proofs of these two lemmas are almost identical we prove only Lemma \ref{AppAABCLLHLim1Deriv} and leave the details of the proof of Lemma \ref{AppAABCLLHLim2Deriv} to the reader.

\bigskip

\begin{proof}[Proof of Lemma \ref{AppAABCLLHLim1Deriv}]
It follows
from \eqref{eq:ABCLLHLimLemPf3} and \eqref{AppAABCLLHLimLempf100} that in order to prove the result it is sufficient to show that 
\begin{align}
\lim_{\epsilon\to0} \bar{E}_{\theta^{\ast}}  \left[  \frac{1}{n} \nabla_{\theta} \log p_{\theta}^{\epsilon}( Y_{1},\ldots, Y_{n}) \right] = \bar{E}_{\theta^{\ast}}  \left[  \frac{1}{n} \nabla_{\theta} \log p_{\theta}( Y_{1},\ldots, Y_{n}) \right]   \label{AppAABCLLHLim1DerivPf700}
\end{align}
for all  $n$ and $\theta$ and hence by \eqref{LemFishInf1}, \eqref{AppAABCLLHLim1DerivPf101} and \eqref{AppAABCLLHLim1DerivPf102} and the accompanying comments and the dominated convergence theorem that 
\begin{align} \label{AppAABCLLHLim1DerivPf2}
& E_{\theta^{\epsilon}}\left[  \nabla_{\theta} \left( \log g_{\theta}^{\epsilon}\left( Y_{k} \vert X_{k} \right)q_{\theta}\left( X_{k-1} , X_{k} \right) \right) \vert Y_{1:n} \right]  \notag \\
& \qquad =  E_{\theta}\left[  \nabla_{\theta} \left( \log g_{\theta} \left( Y_{k} \vert X_{k} \right)q_{\theta}\left( X_{k-1} , X_{k} \right) \right) \vert Y_{1:n} \right] 
\end{align}
$\overline{\mathbb{P}}_{\theta^{\ast}}$ a.s.~for all  $\theta$ and $1 \leq k \leq n$.   Recall that 
\begin{align}
&  E_{\theta^{\epsilon}}\left[  \nabla_{\theta} \left( \log g_{\theta}^{\epsilon}\left( Y_{k} \vert X_{k} \right)q_{\theta}\left( X_{k-1} , X_{k} \right) \right) \vert Y_{1:n} \right]  \label{AppAABCLLHLim1DerivPf3} \\
& =  \frac{ \int_{\mathcal{X}^{n}}   \nabla_{\theta}  \left(  \log g^{\epsilon}_{\theta} \left( Y_{k} \vert x_{k} \right)  q_{\theta}\left( x_{k-1} , x_{k} \right)  \right)  \prod_{i=1}^{n}   \left(  g^{\epsilon}_{\theta} \left( Y_{i} \vert x_{i} \right) q_{\theta}\left( x_{i-1} , x_{i} \right) \right) \mu(dx_{1:n}) }{  \int_{\mathcal{X}^{n}}   \prod_{i=1}^{n}   \left(  g^{\epsilon}_{\theta} \left( Y_{i} \vert x_{i} \right) q_{\theta}\left( x_{i-1} , x_{i} \right) \right) \mu(dx_{1:n})    } \notag
\end{align}
and
\begin{align}
& E_{\theta}\left[  \nabla_{\theta} \left( \log g_{\theta} \left( Y_{k} \vert X_{k} \right)q_{\theta}\left( X_{k-1} , X_{k} \right) \right) \vert Y_{1:n} \right]  \label{AppAABCLLHLim1DerivPf4} \\
& =  \frac{ \int_{\mathcal{X}^{n}}   \nabla_{\theta}  \left(  \log g_{\theta} \left( Y_{k} \vert x_{k} \right)  q_{\theta}\left( x_{k-1} , x_{k} \right)  \right)  \prod_{i=1}^{n}   \left(  g_{\theta} \left( Y_{i} \vert x_{i} \right) q_{\theta}\left( x_{i-1} , x_{i} \right) \right) \mu(dx_{1:n}) }{  \int_{\mathcal{X}^{n}}   \prod_{i=1}^{n}   \left(  g_{\theta} \left( Y_{i} \vert x_{i} \right) q_{\theta}\left( x_{i-1} , x_{i} \right) \right) \mu(dx_{1:n})    } . \notag
\end{align}
Further we have by \eqref{AppAABCLLHLim1DerivPf102} and the accompanying comments that we can use the Lebesgue differentiation theorem (see for example \citep{whezyg1977}) to deduce that for all $x \in \mathcal{X}$ that
\begin{equation}  \label{AppAABCLLHLim1DerivPf7}
 \nabla_{\theta}    g^{\epsilon}_{\theta} \left( Y_{k} \vert x \right) \to  \nabla_{\theta}   g_{\theta} \left( Y_{k} \vert x \right) , \,   g^{\epsilon}_{\theta} \left( Y_{k} \vert x \right) \to   g_{\theta} \left( Y_{k} \vert x \right) 
\end{equation}
$\overline{\mathbb{P}}_{\theta^{\ast}}$ a.s..  It  now follows from assumptions (A2) and (A5), \eqref{AppAABCLLHLim1DerivPf102} etc.~and  \eqref{AppAABCLLHLim1DerivPf7} and the dominated convergence theorem that the numerator and denominator of the quantity in \eqref{AppAABCLLHLim1DerivPf3} converge to respectively the numerator and denominator of the quantity in \eqref{AppAABCLLHLim1DerivPf4}.  Since by assumption (A4) we have that 
\[
\int_{\mathcal{X}^{n}}   \prod_{i=1}^{n}   \left(  g_{\theta} \left( Y_{i} \vert x_{i} \right) q_{\theta}\left( x_{i-1} , x_{i} \right) \right) \mu(dx_{1:n})  > 0
\] 
$\overline{\mathbb{P}}_{\theta^{\ast}}$ a.s.~we obtain  \eqref{AppAABCLLHLim1DerivPf2}.
\end{proof}

\renewcommand{\thesection}{C}
\renewcommand{\theequation}{C-\arabic{equation}}

\section*{Appendix C: Proofs of Theorems \ref{thmABCMLEconsistency}, \ref{thmABCMLEnormality} and \ref{secABCPostBernVMthm}} 

\begin{proof}[Proof of Theorem \ref{thmABCMLEconsistency}]
It follows immediately from Theorem \ref{lemrealanalresstand2} that the first part of Theorem \ref{thmABCMLEconsistency} will hold with  the set $\mathcal{T}^{\epsilon}$  equal to the set of maximisers of $l^{\epsilon} ( \theta )$.  Note that since $l^{\epsilon} ( \theta )$ is continuous and $\Theta$ compact $\mathcal{T}^{\epsilon}$ will always be well defined and non-empty.  Further, \eqref{ABCMLEmaxset} follows from the uniform convergence of $l^{\epsilon} ( \theta )$ to $l^{0} ( \theta )$ and the continuity of the surfaces.

It remains to prove the second part of the theorem.  Suppose now that $\nabla_{\theta}^{2} l^{0} ( \theta^{\ast} )$ is strictly negative definite.   We have from the last part of Theorem \ref{lemrealanalresstand2} that 
\begin{equation} \label{secABCABCBayesPostthmpfeq3}
\lim_{\delta \to 0} \lim_{\epsilon \to 0} \sup_{\vert \theta - \theta^{\ast} \vert \leq \delta} \left\Vert \nabla_{\theta}^{2} l^{\epsilon} (\theta) - \nabla_{\theta}^{2} l^{0} (\theta^{\ast}) \right\Vert = 0 .
\end{equation}
Equation \eqref{secABCABCBayesPostthmpfeq3} implies that there exists some $\delta > 0$ such that for sufficiently small $\epsilon$ the surface $l^{\epsilon} ( \theta )$ has at most one local maximum in the $\delta$ neighbourhood of $\theta^{\ast}$.  The result now follows from \eqref{ABCMLEmaxset}.
\end{proof}

\bigskip

\begin{proof}[Proof of Theorem \ref{thmABCMLEnormality}]
Letting the matrix $I_{\epsilon}$ be equal to $\nabla_{\theta}^{2} l^{\epsilon} (\theta^{\ast, \epsilon})$ it follows from Theorem \ref{secABCABCBayesPostthm} and standard results on the asymptotic normality of the MLE (see for example \citep{doumouryd2004}) that in order to prove Theorem \ref{thmABCMLEnormality} it is sufficient to show that for $\epsilon$ sufficiently small there exists some strictly positive definite matrix $J_{\epsilon}$ such that 
\begin{equation} \label{thmABCMLEnormalityPf1}
\frac{1}{\sqrt{n}} \nabla_{\theta}  \log p_{\theta^{\ast, \epsilon}}^{\epsilon}( Y_{1},\ldots, Y_{n}) \to N ( 0 , J_{\epsilon} )
\end{equation}
and
\begin{equation} \label{thmABCMLEnormalityPf4}
J_{\epsilon} \to I
\end{equation}
as $\epsilon \to 0$ where $I = \nabla^{2}_{\theta} l^{0} ( \theta^{\ast} )$.

We begin by proving \eqref{thmABCMLEnormalityPf1}.  We have by \eqref{eq:ABCLLHLimLemPf3} and \eqref{AppAABCLLHLimLempf100} that 
\begin{align} 
& \frac{1}{\sqrt{n}} \nabla_{\theta}  \log p_{\theta^{\ast, \epsilon}}^{\epsilon}( Y_{1},\ldots, Y_{n}) \notag \\
& = 
\frac{1}{\sqrt{n}}  \sum_{i=1}^{n} R^{\epsilon}_{\theta^{\ast, \epsilon}}(Y_{-\infty:i})  + \frac{1}{\sqrt{n}}  \sum_{i=1}^{n} h^{\epsilon}_{\theta^{\ast, \epsilon}} (Y_{1:i}) - R^{\epsilon}_{\theta^{\ast, \epsilon}}(Y_{-\infty:i}) \label{thmABCMLEnormalityPf2}
\end{align}
where $h^{\epsilon}_{\theta^{\ast, \epsilon}} (Y_{1:i})$ is as defined in \eqref{eq:ABCLLHLimLemPf300}.  We note that it follows from \eqref{doumourydProp41eq} that one can use  similar arguments to those used to deduce \eqref{eq:ABCLLHLimLemPf3} to show that
\begin{equation} \label{thmABCMLEnormalityPf9}
\bar{E}_{\theta^{\ast}} \left[ \sup_{k \geq n} \Big\vert h^{\epsilon}_{\theta} (Y_{-n:0}) - R^{\epsilon}_{\theta}(Y_{-\infty:0})\Big\vert ^{2} 
\right] \leq K \rho^{n} .
\end{equation}
It then follows from \eqref{thmABCMLEnormalityPf9} that 
\begin{equation*}
\bar{E}_{\theta^{\ast}}\left[ h^{\epsilon}_{\theta^{\ast, \epsilon}} (Y_{-n:i}) \vert Y_{-\infty:0} \right]  \xrightarrow{L_{2}} \bar{E}_{\theta^{\ast}} \left[ R^{\epsilon}_{\theta^{\ast, \epsilon}}(Y_{-\infty:i})  \vert Y_{-\infty:0} \right]  
\end{equation*}
and likewise for conditional expectations w.r.t.~$\sigma ( Y_{-\infty:-1} )$ and hence by \eqref{doumourydProp42eq} that there exists some $K$ such that 
\begin{align}
\bar{E}_{\theta^{\ast}} \left[  \left\vert \bar{E}_{\theta^{\ast}} \left[ R^{\epsilon}_{\theta^{\ast, \epsilon}}(Y_{-\infty:i})  \vert Y_{-\infty:0} \right]  -  \bar{E}_{\theta^{\ast}}  \left[ R^{\epsilon}_{\theta^{\ast, \epsilon}}(Y_{-\infty:i})  \vert Y_{-\infty:-1} \right]  \right\vert^{2}  \right]  \leq  K  \rho^{i}  \label{thmABCMLEnormalityPf3}
\end{align}
for all $i$.  Equation \eqref{thmABCMLEnormalityPf3} immediately implies that the sequence of random variables $R^{\epsilon}_{\theta^{\ast, \epsilon}}(Y_{-\infty:0}), R^{\epsilon}_{\theta^{\ast, \epsilon}}(Y_{-\infty:1}) , \ldots $ satisfies the conditions of Theorem 5 in \citep{vol1993} and hence we have that
\begin{align}   \label{thmABCMLEnormalityPf5}
\frac{1}{\sqrt{n}}  \sum_{i=1}^{n} R^{\epsilon}_{\theta^{\ast, \epsilon}}(Y_{-\infty:i})  \xrightarrow{\text{weakly}}  N(0 , J_{\epsilon})
\end{align}
where 
\begin{align}  \label{thmABCMLEnormalityPf6}
 J_{\epsilon}  =  \lim_{n \to \infty}   \bar{E}_{\theta^{\ast}}  \left[  \left(  \frac{1}{\sqrt{n}}  \sum_{i=1}^{n} R^{\epsilon}_{\theta^{\ast, \epsilon}}(Y_{-\infty:i})  \right)^{2}  \right]  .
\end{align}
Finally we note that it follows from \eqref{AppAABCLLHLimLempf100}, Markov's inequality and the Borel-Cantelli lemma that 
\begin{equation}  \label{thmABCMLEnormalityPf7}
\bar{\mathbb{P}}_{\theta^{\ast}} \left( \left\vert h^{\epsilon}_{\theta^{\ast, \epsilon}} (Y_{1:i}) - R^{\epsilon}_{\theta^{\ast, \epsilon}}(Y_{-\infty:i})  \right\vert >  \frac{1}{i^{2}} \text{ i.o. } \right) = 0.
\end{equation}
Equation \eqref{thmABCMLEnormalityPf1} now follows from \eqref{thmABCMLEnormalityPf5} and \eqref{thmABCMLEnormalityPf7}.

To complete the proof of the theorem it remains to prove \eqref{thmABCMLEnormalityPf4}.  It immediately follows from  \eqref{eq:ABCLLHLimLemPf300}, \eqref{thmABCMLEnormalityPf9} and \eqref{thmABCMLEnormalityPf6} that for all $\epsilon$
\begin{align}  \label{thmABCMLEnormalityPf8}
 J_{\epsilon}  =  \lim_{n \to \infty}   \bar{E}_{\theta^{\ast}} \left[ \frac{1}{n} \nabla_{\theta}  \log p_{\theta^{\ast, \epsilon}}^{\epsilon}( Y_{1},\ldots, Y_{n})^{2}  \right]  
\end{align}
where the convergence is uniform in $n$.  Next we note that by a simple application of the Fisher identity (see for example \citep{doumouryd2004}) that 
\begin{align*}
\bar{E}_{\theta^{\ast}} \left[  \nabla_{\theta}^{2}  \log p_{\theta^{\ast}} ( Y_{1},\ldots, Y_{n})  \right]  =   \bar{E}_{\theta^{\ast}} \left[  \nabla_{\theta}  \log p_{\theta^{\ast}} ( Y_{1},\ldots, Y_{n})^{2}  \right]
\end{align*}
and thus by \eqref{eq:ABCLLHLimLemPf6} and Lemma \ref{AppAABCLLHLimLem} that
\begin{align}
I =  \lim_{n \to \infty}  \bar{E}_{\theta^{\ast}} \left[  \frac{1}{n} \nabla_{\theta}  \log p_{\theta^{\ast}}^{\epsilon}( Y_{1},\ldots, Y_{n})^{2}  \right]   \label{thmABCMLEnormalityPf10}
\end{align}
In order to complete the proof of \eqref{thmABCMLEnormalityPf4} it is thus sufficient, by \eqref{thmABCMLEnormalityPf8} and \eqref{thmABCMLEnormalityPf10}, to show that for all $n$
\begin{align}  \label{thmABCMLEnormalityPf11}
 \lim_{\epsilon \to \infty}   \bar{E}_{\theta^{\ast}} \left[ \frac{1}{n} \nabla_{\theta}  \log p_{\theta^{\ast, \epsilon}}^{\epsilon}( Y_{1},\ldots, Y_{n})^{2}  \right]   =   \bar{E}_{\theta^{\ast}} \left[  \frac{1}{n} \nabla_{\theta}  \log p_{\theta^{\ast}} ( Y_{1},\ldots, Y_{n})^{2}  \right]   .
\end{align}
Finally we note that \eqref{thmABCMLEnormalityPf11} can be proved in exactly the same way as \eqref{AppAABCLLHLim1DerivPf7} in the proof of Lemma \ref{AppAABCLLHLim1Deriv}.  In order to this we need to show that 
\begin{equation}   \label{thmABCMLEnormalityPf12} 
 \nabla_{\theta}    g^{\epsilon}_{\theta^{\ast, \epsilon}} \left( Y_{k} \vert x \right) \to  \nabla_{\theta}   g_{\theta^{\ast}} \left( Y_{k} \vert x \right) , \,   g^{\epsilon}_{\theta^{\ast, \epsilon}} \left( Y_{k} \vert x \right) \to   g_{\theta^{\ast}} \left( Y_{k} \vert x \right) 
\end{equation}
 as $\epsilon \to 0$.  However \eqref{thmABCMLEnormalityPf12}  follows from \eqref{AppAABCLLHLim1DerivPf7} and the fact that by assumptions (A2) and (A6) we have that $\mathbb{P}_{\theta^{\ast}}$ a.s.~the functions $\nabla_{\theta}    g_{\theta} \left( Y_{k} + z \vert x \right)$ and $g_{\theta} \left( Y_{k} + z \vert x \right) $ are uniformly Lipschitz (as functions of $\theta$) for all $z \in B^{\epsilon}_{0}$.
\end{proof}

\bigskip

\begin{proof}[Proof of Theorem \ref{secABCPostBernVMthm}]
The proof of this result follows from  standard Bernstein-Von Mises type arguments, see for example \citep{borkalrao1971}.
\end{proof}

\renewcommand{\thesection}{D}
\renewcommand{\theequation}{D-\arabic{equation}}

\section*{Appendix D: Proofs of Theorems \ref{secABCsmallepsthmGenCase} and \ref{secABCsmallepsthmLeb}}

A central role in the proof of Theorem \ref{secABCsmallepsthmGenCase} will be played by the following time inhomogeneous versions of the perturbed HMM \eqref{alg1pertHMM}.

Suppose that one has a collection of HMMs parametrised by some parameter vector $\theta \in \Theta$ and that for each value of $\theta$ the conditional laws and transition kernels of the corresponding HMM have densities $g_{\theta} ( y \vert x)$ and $q_{\theta} ( x , x^{\prime} )$ respectively.  Given some $\theta \in \Theta$ and $\epsilon > 0$ define the HMM  $\left\{ X^{\epsilon,+}_{k}, Y^{\epsilon,+}_{k} \right\}_{k \in \mathbb{Z} }$ by
\begin{align} \label{pertHMM+}
X^{\epsilon,+}_{k}, Y^{\epsilon,+}_{k} = X_{k}, Y_{k} \text{ for all } k \leq 0; \quad X^{\epsilon,+}_{k}, Y^{\epsilon,+}_{k} = X_{k}, Y_{k} + \epsilon Z_{k} \text{ o.w. }
\end{align}
where $\left\{ X_{k}, Y_{k} \right\}_{k \in \left\{ \mathbb{Z} \right\}}$ is the original HMM and $\left\{ Z_{k}  \right\}_{k \geq 0  }$ is a collection of i.i.d.~$\mathcal{U}_{B^{1}_{0}}$ random variables.  Similarly define the HMM $\left\{ X^{\epsilon,-}_{k}, Y^{\epsilon,-}_{k} \right\}_{k \in \mathbb{Z} }$ by
\begin{align} \label{pertHMM-}
X^{\epsilon,-}_{k}, Y^{\epsilon,-}_{k} = X_{k}, Y_{k} \text{ for all } k < 0; \quad X^{\epsilon,-}_{k}, Y^{\epsilon,-}_{k} = X_{k}, Y_{k} + \epsilon Z_{k} \text{ o.w.}.
\end{align}
Clearly the transition kernels of the HMMs \eqref{pertHMM+} and \eqref{pertHMM-} are equal to $q_{\theta} \left( x , x^{\prime} \right)$ and the conditional densities of the observed state are equal to 
\begin{equation} \label{LemGradLLHDiff1}
g^{\epsilon, +}_{\theta , k} \left( y \vert x \right) = 
\left\{
\begin{array}{cc}
g^{\epsilon}_{\theta } \left( y \vert x \right) & \text{ if  } \, k > 0 \\
g_{\theta } \left( y \vert x \right) & \text{ otherwise }
\end{array}
\right. 
\end{equation}
and
\begin{equation} \label{LemGradLLHDiff2}
g^{\epsilon, -}_{\theta , k} \left( y \vert x \right) = 
\left\{
\begin{array}{cc}
g^{\epsilon}_{\theta } \left( y \vert x \right) & \text{ if  } \, k \geq 0 \\
g_{\theta } \left( y \vert x \right) & \text{ otherwise }
\end{array}
\right. 
\end{equation}
respectively.  

Let $p_{\theta^{\epsilon,+}} ( \cdots )$, $\mathbb{P}_{\theta^{\epsilon,+}} \left( \cdot \right)$, $E_{\theta^{\epsilon,+}} \left[ \cdot \right]$, $E_{\theta^{\epsilon,+}} \left[ \cdot \vert \cdot \right]$ and $\mathbb{P}_{\theta^{\epsilon,+}} \left( \cdot \vert \cdot \right)$ and  $p_{\theta^{\epsilon,-}} ( \cdots )$, $\mathbb{P}_{\theta^{\epsilon,-}} \left( \cdot \right)$, $E_{\theta^{\epsilon,-}} \left[ \cdot \right]$, $E_{\theta^{\epsilon,-}} \left[ \cdot \vert \cdot \right]$ and $\mathbb{P}_{\theta^{\epsilon,-}} \left( \cdot \vert \cdot \right)$ denote the likelihood functions, laws and expectation, conditional expectation and conditional probability operators w.r.t. to the laws of \eqref{pertHMM+} and \eqref{pertHMM-}.  It follows by definition that
\begin{equation} \label{pertHMMsLLHrel}
\begin{gathered}
p_{\theta} ( y_{1} , \ldots , y_{n})  =  p_{\theta^{\epsilon,+}} ( Y_{-n+1} = y_{1},\ldots, Y_{0} = y_{n} ) \\
p^{\epsilon}_{\theta} ( y_{1} , \ldots , y_{n})  =  p_{\theta^{\epsilon,-}} ( Y_{0} = y_{1},\ldots, Y_{n-1} = y_{n} ) \\
p_{\theta^{\epsilon,+}} ( Y_{-k+1} = y_{-k+1},\ldots, Y_{n} = y_{n} ) = p_{\theta^{\epsilon,-}} ( Y_{-k} = y_{-k+1},\ldots, Y_{n-1} = y_{n} )  .
\end{gathered}
\end{equation}
Recall that by \eqref{eq:ABCLLHLimLemPf3} and \eqref{AppAABCLLHLimLempf100} we have that 
\begin{align}  
& \nabla_{\theta} l \left(\theta \right)-\nabla_{\theta} l^{\epsilon} \left(\theta \right) \label{LemGradLLHDiffPf1} \\
& \qquad   = \lim_{n \to \infty} \frac{1}{n} \left(  \bar{E}_{\theta^{\ast}} \left[  \nabla_{\theta} \log p_{\theta} ( Y_{1},\ldots, Y_{n}) \right] - \bar{E}_{\theta^{\ast}}  \left[  \nabla_{\theta} \log p_{\theta}^{\epsilon} ( Y_{1},\ldots, Y_{n}) \right]  \right)   \notag
 \end{align}
and thus by \eqref{pertHMMsLLHrel} and \eqref{LemGradLLHDiffPf1} that we have the telescoping sum
\begin{align} \label{LemGradLLHDiffPf5}
\lefteqn{ \nabla_{\theta} l \left(\theta \right)-\nabla_{\theta} l^{\epsilon} \left(\theta \right) = \lim_{n \to \infty} \frac{1}{n} \sum_{i=1}^{n} \Big(  \bar{E}_{\theta^{\ast}} \left[  \nabla_{\theta} \log p_{\theta^{\epsilon,+}} ( Y_{-n+i},\ldots, Y_{i-1}) \right] }\notag \\
&& \qquad \qquad \qquad \qquad  - \bar{E}_{\theta^{\ast}} \left[  \nabla_{\theta} \log p_{\theta^{\epsilon,-}} ( Y_{-n+i},\ldots, Y_{i-1}) \right]  \Big)  .
 \end{align}
We now note that Theorems \ref{secABCsmallepsthmGenCase} and \ref{secABCsmallepsthmLeb} follow immediately from \eqref{LemGradLLHDiffPf5} and the following lemma
\begin{lemma} \label{LemGradLLHRateAux}
Suppose that assumptions (A2)-(A7) hold for a collection of HMMs parametrised by some vector $\theta \in \Theta$.  Then there exists a finite constant $K$ such that for all $\epsilon > 0$ and integers $k, n$
\begin{align}   
& \Big\vert  E_{\theta^{\ast}}\left[  \nabla_{\theta} \log p_{\theta^{\epsilon,+}} ( Y_{-k},\ldots, Y_{n}) \right]   - E_{\theta^{\ast}}\left[  \nabla_{\theta} \log p_{\theta^{\epsilon,-}} ( Y_{-k},\ldots, Y_{n}) \right]    \Big\vert \label{LemGradLLHRate5} \\
& \qquad \qquad \qquad \qquad   \leq K    \epsilon .  \notag
\end{align}
Furthermore suppose that $\nu$ is Lebesgue measure and that assumption (A8) also holds.  Then
\begin{align}   
& \Big\vert  E_{\theta^{\ast}}\left[  \nabla_{\theta} \log p_{\theta^{\epsilon,+}} ( Y_{-k},\ldots, Y_{n}) \right]   - E_{\theta^{\ast}}\left[  \nabla_{\theta} \log p_{\theta^{\epsilon,-}} ( Y_{-k},\ldots, Y_{n}) \right]    \Big\vert \label{LemGradLLHRate5} \\
& \qquad \qquad \qquad \qquad   \leq K    \epsilon^{2} .  \notag
\end{align}
\end{lemma}

\begin{proof}
We shall prove only the first part of the lemma, the proof of the second part being almost identical.  Clearly analogous expressions to \eqref{LemFishInf1} hold for the HMMs \eqref{pertHMM+} and \eqref{pertHMM-} and thus in particular we have that the term on the left hand side of \eqref{LemGradLLHRate5} is bounded by
\begin{align}   
& \left\vert \bar{E}_{\theta^{\ast}}  \left[ \sum_{i=-k}^{n}  E_{\theta^{\epsilon,+}}\left[  \nabla_{\theta}   \log g_{\theta,i}^{\epsilon,+}\left( Y_{i} \vert X_{i} \right)q_{\theta}\left( X_{i-1} , X_{i} \right) \Big\vert  Y_{-k:n}\right] -  \right. \right.  \notag  \\
& \qquad \qquad \left. \left. \sum_{i=-k}^{n}  E_{\theta^{\epsilon,-}}\left[ \nabla_{\theta}   \log g_{\theta,i}^{\epsilon,-}\left( Y_{i} \vert X_{i} \right)q_{\theta}\left( X_{i-1} , X_{i} \right) \Big\vert  Y_{-k:n}\right]  \right]    \right\vert  \label{LemGradLLHRateAuxPf1}
\end{align}
where $g_{\theta,i}^{\epsilon,+}$ and $g_{\theta,i}^{\epsilon,-}$ are as in \eqref{LemGradLLHDiff1} and \eqref{LemGradLLHDiff2}.  Using the identity 
\begin{align*} 
&  \frac{  \nabla_{\theta}  g_{\theta} \left( Y_{0} \vert X_{0} \right)  }{  g_{\theta} \left( Y_{0} \vert X_{0} \right)  }  -   \frac{  \nabla_{\theta}  g_{\theta}^{\epsilon} \left( Y_{0} \vert X_{0} \right)  }{  g_{\theta}^{\epsilon} \left( Y_{0} \vert X_{0} \right)  }  =  \left( \frac{  \nabla_{\theta}  g_{\theta} \left( Y_{0} \vert X_{0} \right)  }{  g_{\theta} \left( Y_{0} \vert X_{0} \right)  }  -   \frac{  \nabla_{\theta}  g_{\theta}^{\epsilon} \left( Y_{0} \vert X_{0} \right)  }{  g_{\theta}\left( Y_{0} \vert X_{0} \right)  }  \right)  \\  
& +  \frac{  \nabla_{\theta}  g_{\theta}^{\epsilon} \left( Y_{0} \vert X_{0} \right)  }{  g_{\theta}^{\epsilon} \left( Y_{0} \vert X_{0} \right)  }  \left( \frac{   g_{\theta}^{\epsilon}  \left( Y_{0} \vert X_{0} \right) - g_{\theta} \left( Y_{0} \vert X_{0} \right)  }{  g_{\theta} \left( Y_{0} \vert X_{0} \right)  }  \right) 
\end{align*}
it is clear from \eqref{AppAABCLLHLim1DerivPf101} and assumptions (A6) and (A7) that there exists some $K^{\prime}$ such that 
\begin{align}   
& \left\vert  \bar{E}_{\theta^{\ast}}  \left[   E_{\theta^{\epsilon,+}}\left[  \nabla_{\theta}   \log g_{\theta,0}^{\epsilon,+}\left( Y_{i} \vert X_{i} \right)q_{\theta}\left( X_{i-1} , X_{i} \right) \Big\vert  Y_{-k:n}\right] -  \right. \right.  \notag  \\
& \qquad \qquad \left. \left. E_{\theta^{\epsilon,+}}\left[ \nabla_{\theta}   \log g_{\theta,0}^{\epsilon,-}\left( Y_{i} \vert X_{i} \right)q_{\theta}\left( X_{i-1} , X_{i} \right) \Big\vert  Y_{-k:n}\right]  \right]    \right\vert  \leq K^{\prime} \epsilon. 
\end{align}
It then follows from the definitions of $g_{\theta,i}^{\epsilon,+}$ and $g_{\theta,i}^{\epsilon,-}$ and from assumptions (A2)-(A6) that in order to derive \eqref{LemGradLLHRate5} from \eqref{LemGradLLHRateAuxPf1} it is sufficient to show that for all $i,k,n$
\begin{align}
\bar{E}_{\theta^{\ast}}  \left[ \left\Vert  \mathbb{P}_{\theta^{\epsilon,+}} \left( X_{i-1}, X_{i}  \big\vert  Y_{-k:n} \right)  - \mathbb{P}_{\theta^{\epsilon,-}} \left( X_{i-1}, X_{i}  \big\vert  Y_{-k:n} \right)   \right\Vert_{TV}  \right] \leq K^{\prime \prime}  \epsilon \rho^{ \vert i \vert} \label{LemGradLLHRateAuxPf2}
\end{align}
for some $K^{ \prime \prime}$.  

We first note that it follows from standard results concerning uniformly mixing Markov chains, see for example \citep{caprydmou2005, del2004} that  there exist some $K$ and  $0<\rho<1$ such that for all $i$  and $x,x^{\prime} \in \mathcal{X}$ one has that
\begin{align}  
& \Big\Vert \mathbb{P}_{\theta^{\epsilon,+}} \left( X_{i-1}, X_{i} \vert X_{0} = x \right)   -  \mathbb{P}_{\theta^{\epsilon,+}} \left( X_{i-1}, X_{i} \vert X_{0} = x^{\prime} \right)   \Big\Vert_{TV}   \leq K \rho^{\vert i \vert - 1 }  .\label{eq:lemUnifErgDual1}
\end{align}
Since by definition we have that the marginal laws of $\mathbb{P}_{\theta^{\epsilon,+}} \left( Y_{-\infty:-1 ; 1:\infty} \right)$ and $\mathbb{P}_{\theta^{\epsilon,-}} \left( Y_{-\infty:-1 ; 1:\infty} \right)$ are equal  it follows that in order to prove \eqref{LemGradLLHRateAuxPf2} it is sufficient to only prove it for the case $i=0$.

To prove \eqref{LemGradLLHRateAuxPf2}  for $i=0$ we shall make use of the following simple identities.  For any $\phi \in L_{\infty}$
\begin{equation}  \label{LemGradLLHRateAux3}
\begin{gathered}
 E_{\theta^{\epsilon,+}} \left[  \phi (X_{0}) \vert Y_{\infty  : \infty} \right]  = \frac{  E_{\theta^{\epsilon,+}} \left[ \phi (X_{0}) g_{\theta} \left( Y_{0} \vert X_{0} \right) \vert Y_{\infty : -1 ; 1 : \infty} \right]  }{  E_{\theta^{\epsilon,+}} \left[   g_{\theta} \left( Y_{0} \vert X_{0} \right)  \vert Y_{\infty : -1 ; 1 : \infty} \right]  } \\
  E_{\theta^{\epsilon,-}} \left[  \phi (X_{0}) \vert Y_{\infty  : \infty} \right]  = \frac{  E_{\theta^{\epsilon,-}} \left[ \phi (X_{0}) g_{\theta}^{\epsilon} \left( Y_{0} \vert X_{0} \right) \vert Y_{\infty : -1 ; 1 : \infty} \right]  }{  E_{\theta^{\epsilon,-}} \left[   g_{\theta}^{\epsilon} \left( Y_{0} \vert X_{0} \right)  \vert Y_{\infty : -1 ; 1 : \infty} \right]  }  .
\end{gathered}
\end{equation}
It then follows from \eqref{LemGradLLHRateAux3} using basic algebra  that 
\begin{align}
& \sup_{\phi : \left\Vert \phi \right\Vert_{\infty} \leq 1 } \bar{E}_{\theta^{\ast}}   \bigg[  \Big\vert E_{\theta^{\epsilon,+}} \left[  \phi (X_{0}) \vert Y_{\infty  : \infty} \right]  -   E_{\theta^{\epsilon,-}} \left[  \phi (X_{0}) \vert Y_{\infty  : \infty} \right]  \Big\vert \bigg] \notag \\
& \leq    \bar{E}_{\theta^{\ast}}  \left[   \frac{   E_{\theta^{\epsilon,+}} \left[  \left\vert   g_{\theta}  - g_{\theta}^{\epsilon} \right\vert  \vert Y_{\infty : -1 ; 1 : \infty} \right]     }{  E_{\theta^{\epsilon,+}} \left[   g_{\theta}  \vert Y_{\infty : -1 ; 1 : \infty} \right]  }     + \frac{  E_{\theta^{\epsilon,-}} \left[   \left\vert   g_{\theta}  - g_{\theta}^{\epsilon} \right\vert   \vert Y_{\infty : -1 ; 1 : \infty} \right] }{  E_{\theta^{\epsilon,-}} \left[   g_{\theta}  \vert Y_{\infty : -1 ; 1 : \infty} \right]  }   \right]  .\label{LemGradLLHRateAux5} 
\end{align}
The result now follows follows immediately from \eqref{LemGradLLHRateAux5}  and assumption (A7).
\end{proof}

\bigskip

\bibliographystyle{plainnat}
\bibliography{references}

\end{document}